\documentclass[12pt]{amsart}
\usepackage{epsfig}
\usepackage{amssymb, amsmath}
\usepackage{amsfonts,graphics,epsfig}
\usepackage{mathrsfs}
\usepackage{pb-diagram, pb-xy}
\usepackage[all]{xy}
\usepackage{comment}
\usepackage{tikz}
\usepackage{setspace}
\usepackage[left=1.5in,right=1.5in,top=1.5in,bottom=1.5in]{geometry}
\usepackage{color}   
\usepackage{hyperref}
\hypersetup{
    colorlinks=true, 
    linktoc=all,     
    linkcolor=red,
	citecolor=blue,  
}
\usepackage{tcolorbox}

\newcommand{\mbR}{\mathbb{R}}
\newcommand{\mbC}{\mathbb{C}}

\newcommand{\mbQ}{\mathbb{Q}}
\newcommand{\mbK}{\mathbb{K}}

\def\mbN{\mathbb{N}}
\def\mbP{\mathbb{P}}

\newcommand{\<}{\leq}
\def\>{\geq}

\def\ve{\varepsilon}

\def\subset{\subseteq}

\newcommand{\lrd}{\lfloor}
\newcommand{\rrd}{\rfloor}

\newcommand{\num}{\equiv}

\def\mcA{\mathcal{A}}
\def\mcO{\mathcal{O}}

\def\mcC{\mathcal{C}}
\def\mcD{\mathcal{D}}
\def\mcE{\mathcal{E}}

\def\mcL{\mathcal{L}}

\def\mcN{\mathcal{N}}
\def\mcP{\mathcal{P}}
\def\mcR{\mathcal{R}}
\def\mcS{\mathcal{S}}
\def\mcT{\mathcal{T}}

\def\mcW{\mathcal{W}}

\newcommand{\rtmap}{\dashrightarrow}

\newtheorem{theorem}{Theorem}[section]
\newtheorem{lemma}[theorem]{Lemma}
\newtheorem{proposition}[theorem]{Proposition}
\newtheorem{corollary}[theorem]{Corollary}
\newtheorem{conjecture}[theorem]{Conjecture}

\theoremstyle{remark}
\newtheorem{remark}[theorem]{Remark}
\theoremstyle{definition}
\newtheorem{definition}[theorem]{Definition}
\theoremstyle{definition}

\numberwithin{equation}{section}
\theoremstyle{definition}

\def\Supp{\operatorname{Supp}}
\def\dim{\operatorname{dim}}
\def\codim{\operatorname{codim}}
\def\chr{\operatorname{char}}
\def\Ex{\operatorname{Ex}}

\def\WDiv{\operatorname{WDiv}}

\def\NE{\overline{\operatorname{NE}}}

\def\max{\operatorname{max}}

\def\Eff{\operatorname{\overline{Eff}}}

\def\NM{\operatorname{\overline{NM}}}
\def\SNM{\operatorname{\overline{SNM}}}
\def\NF{\operatorname{\overline{NF}}}
\def\Univ{\operatorname{Univ}}
\def\Hom{\operatorname{Hom}}
\def\ev{\operatorname{ev}}
\def\pt{\operatorname{pt}}
\def\vol{\operatorname{vol}}

\author{Omprokash Das}
\address{Department of Mathematics\\
University of California, Los Angeles\\
520 Portola Plaza\\
Math Sciences Building 6363.}
\email{das@math.ucla.edu, omprokash@gmail.com}

\date{}

\begin{document}
	\keywords{Minimal Model Program, Nef Curves, Movable Curves, Positive Characteristic, Characteristic $p>5$, $3$-folds, Finiteness of LMMP, Batyrev's Conjecture}
	\subjclass[2010]{14E30, 14E05, 14E99}
\title[Finiteness of LMM and Nef curves on $3$-folds in $\chr p>5$]{Finiteness of Log Minimal Models and Nef curves on $3$-folds in characteristic $p>5$}
\maketitle

\begin{abstract}
In this article we prove a finiteness result on the number of log minimal models for $3$-folds in $\chr p>5$. We then use this result to prove a version of Batyrev's conjecture on the structure of nef cone of curves on $3$-folds in characteristic $p>5$. We also give a proof of the same conjecture in full generality in characteristic $0$. We further verify that the duality of movable curves and pseudo-effective divisors hold in arbitrary characteristic. We then give a criterion for the pseudo-effectiveness of the canonical divisor $K_X$ of a smooth projective variety in arbitrary characteristic in terms of the existence of a family of rational curves on $X$.
\end{abstract}

\tableofcontents

\section{Introduction}
Lots of progress have been made recently on the log minimal model program for $3$-folds in characteristic $p>5$, see \cite{HX15}, \cite{Bir16}, \cite{BW17} and \cite{Wal17}. One of the things that is not treated in these papers is the finiteness of the number of log minimal models. A partial answer was given in \cite[Theorem 1.4]{BW17}. Here we show that a stronger finiteness result (Theorem \ref{thm:finite-mmp}) analogous to \cite[Corollary 1.1.5]{BCHM10} holds on $3$-folds in $\chr p>5$. We then give some applications of this result (Corollary \ref{cor:adjoint-rings}).\\

On the second part of the paper we work with the nef and movable cone of curves and pseudo-effective divisors. First we verify that a famous theorem \cite[Theorem 2.2]{BDPP13} on the duality of \emph{strongly movable curves} and pseudo-effective divisors hold in positive characteristic in arbitrary dimension (Theorem \ref{thm:cone-duality}). We then give some applications of this result (Theorem \ref{thm:covering-rational-curves} and Corollary \ref{cor:uniruled-criterion}). Finally we focus our attention to Batyrev's conjecture on the structure of nef cone of curves.
\begin{conjecture}\emph{\cite[Conjecture 4.4]{Bat92}}\label{con:batyrev}
	Let $(X, \Delta)$ be a projective KLT pair. Then there are countably many $(K_X+\Delta)$-negative movable curves $C_i$ such that 
	\[\NE(X)_{K_X+\Delta\>0}+\NM(X)=\NE(X)_{K_X+\Delta\>0}+\sum\mbR_{\>0}[C_i].    \]
The rays $\mbR_{\>0}[C_i]$ only accumulate along the hyperplane $(K_X+\Delta)^\bot$.\\	
	\end{conjecture}
This conjecture is one of the main outstanding conjecture in this direction. We prove a version of this conjecture on $3$-folds in $\chr p>5$ (Theorem \ref{thm:nef-cone-structure}). We also give a proof of this conjecture in full generality over the field of complex numbers (Theorem \ref{thm:nef-cone-finiteness}). Batyrev proved this conjecture for terminal $3$-folds defined over $\mbC$. However, his proof contained an error which was later rectified by Araujo in \cite{Ara10}. Using sophisticated tools from \cite{BCHM10} she laid down a clear path towards the proof of the higher dimensional version of the conjecture. Her results were then sharpened later by Lehmann in \cite{Leh12} again using the tools from \cite{BCHM10}. We follow the general strategy as in \cite{Leh12} in proving this conjecture in positive characteristic, and the finiteness of log minimal models (Theorem \ref{thm:finite-mmp}) becomes indispensable in this process.\\

 The following theorem is the positive characteristic ($\chr p>5$) analog of \cite[Corollary 1.1.5]{BCHM10} on $3$-folds. This result is interesting on its own and should be useful in the future. 
\begin{theorem}\emph{(Finiteness of log minimal models)}
\label{thm:finite-mmp}	
Let $\pi:X\to U$ be a projective contraction between two normal projective varieties with $\dim X=3$ and $\chr p>5$. Let $V$ be a finite dimensional affine subspace of $\WDiv_\mbR(X)$ which is defined over the rationals. Suppose that there is a divisor $0\<\Delta_0\in V$ such that $K_X+\Delta_0$ is KLT. Let $A$ be a general ample $\mbQ$-divisor over $U$, which has no components in common with any element of $V$. Then the followings hold:
\begin{enumerate}
	\item There  are finitely many birational contractions $\phi_i:X\rtmap Y_i$ over $U$, $1\<i\<m$ such that
	\[\mcE_{A, \pi}(V)=\bigcup_{i=1}^m\mcW_i, \]
	where each $\mcW_i=\mcW_{\phi_i, A, \pi}(V)$ is a rational polytope. Moreover, if $\phi:X\rtmap Y$ is a log minimal model of $(X, \Delta)$ over $U$, for some $\Delta\in\mcE_{A,\pi}(V)$, then $\phi=\phi_i$, for some $1\<i\<m$.\\
	\item There are finitely many rational maps $\psi_j:X\rtmap Z_j$ over $U$, $1\<j\<n$ which partition $\mcE_{A, \pi}(V)$ into the subsets $\mcA_j=\mcA_{\psi_j, A, \pi}(V)$.\\
	\item For every $1\<i\<m$ there is a $1\<j\<n$ and a morphism $f_{i, j}:Y_i\to Z_j$ such that $\mcW_i\subset\bar{\mcA_j}$.\\
	\item In particular, $\mcE_{A, \pi}(V)$ is a rational polytope and each $\bar{\mcA_j}$ is a finite union of rational polytopes.\\
	\end{enumerate}	
	\end{theorem}

A direct consequence of Theorem \ref{thm:finite-mmp} is that the ring of adjoint divisors is finitely generated:
\begin{corollary}\label{cor:adjoint-rings}
	Let $\pi:X\to U$ be a projective contraction between two normal projective varieties with $\dim X=3$ and $\chr p>5$. Fix an ample$/U$ $\mbQ$-divisor $A\>0$ on $X$. Let $\Delta_i=A+B_i$ for some $\mbQ$-divisors $B_1, B_2,\ldots, B_k\>0$. Assume that $D_i=K_X+\Delta_i$ is DLT and $\mbQ$-Cartier for all $1\<i\<k$. Then the adjoint ring
	\[
		\mcR(\pi, D^\bullet)=\bigoplus_{m_i\in\mbN, 1\<i\<k}\pi_*\mcO_X\left(\left\lrd\sum m_iD_i\right\rrd\right)
	\]
is a finitely generated $\mcO_U$-algebra.\\	 
\end{corollary}

Next we verify that a positive characteristic analog of \cite[Theorem 2.2]{BDPP13} holds in arbitrary dimension.
\begin{theorem}\label{thm:cone-duality}
Let $X$ be a projective variety defined over an algebraically closed field of arbitrary characteristic. Then the cone $\Eff(X)$ of pseudo-effective divisors is dual to the cone $\SNM(X)$ of strongly movable curves.\\ 	
\end{theorem}

\begin{remark}
	We note that this theorem is believed to be known among the experts, while the actual proof in full generality (the $\chr p>0$ version) never appeared in any literature as far as we know. In the case when $X$ is a \emph{smooth} projective variety in characteristic $p>0$, the proof of Theorem \ref{thm:cone-duality} is outlined by Fulger and Lehmann in \cite[Theorem 2.22]{FL17} based on the proof of \cite{BDPP13}. So if we assume the existence of \emph{resolution of singularities} in characteristic $p>0$, then Theorem \ref{thm:cone-duality} will be a formal consequence of \cite[Theorem 2.22]{FL17}. In particular, when $\dim X\<3$, Theorem \ref{thm:cone-duality} follows from Fulger and Lehmann's result.\\
	In our proof we verify that the proof presented in \cite[Theorem 11.4.19]{Laz04b} (which does not assume resolution of singularities) works in positive characteristic with the help of Takagi's Fujita Approximation Theorem \cite{Tak07} in characteristic $p>0$.\\   
\end{remark}

As an application of Theorem \ref{thm:cone-duality} we give a criterion for the pseudo-effectiveness of $K_X$ in terms of the existence of an \emph{algebraic family of rational curves} on $X$. This answers partially a question of Campana \cite[Question 12.1]{Cam16} in arbitrary characteristic.
\begin{theorem}\label{thm:covering-rational-curves}
	Let $X$ be a smooth projective variety defined over an algebraically closed field of arbitrary characteristic. Then $K_X$ is not pseudo-effective if and only if there exists an algebraic family of $K_X$-negative rational curves covering a dense subset of $X$.\\
	\end{theorem}

An immediate corollary of this result is the following sufficient condition for uniruledness in positive characteristic.
\begin{corollary}\label{cor:uniruled-criterion}
	Let $X$ be a smooth projective variety defined over an algebraically closed field of arbitrary characteristic. If $K_X$ is not pseudo-effective, then $X$ is uniruled.
\end{corollary}
Note that the converse of this statement is true in $\chr 0$ and false in $\chr p>0$.\\

We then prove a version of Batyrev's conjecture on the structure of nef cone of curves on $3$-folds in characteristic $p>5$.
\begin{theorem}\label{thm:nef-cone-structure}
 Let $(X, \Delta)$ be a projective DLT pair of dimension $3$ in char $p>5$. Then there are countably many $(K_X+\Delta)$-negative movable curves $C_i$ such that
\[\NE(X)_{K_X+\Delta\>0}+\NM(X)=\NE(X)_{K_X+\Delta\>0}+\overline{\sum\mbR_{\>0}[C_i]}.\]

The rays $\mbR_{\>0}[C_i]$ only accumulate along the hyperplanes which support both $\NM(X)$ and $\NE(X)_{K_X+\Delta\>0}$.\\ 
\end{theorem}
 
Finally we prove the finiteness of co-extremal rays of the nef cone of curves for varieties of arbitrary dimension over $\mbC$. We use Koll\'ar's effective base-point free theorem and the boundedness of \emph{$\ve$-log canonical} log Fano varieties (formerly known as the `BAB Conjecture') recently proved by Birkar \cite{Bir16sep}. This gives a complete proof of Batyrev's conjecture in full generality in characteristic $0$. We note that the finiteness of co-extremal rays was proved in \cite{Ara10} for terminal $3$-folds over $\mbC$, and a weaker version of it is also proved in \cite{Leh12} on arbitrary dimensions over $\mbC$.\\
\begin{theorem}\label{thm:nef-cone-finiteness}
	Let $(X, \Delta\>0)$ be a projective KLT pair over $\mbC$. 
	\begin{enumerate}
	\item There are countably many $(K_X+\Delta)$-negative movable curves $C_i$ such that 
	\[\NE(X)_{K_X+\Delta\>0}+\NM(X)=\NE(X)_{K_X+\Delta\>0}+\sum\mbR_{\>0}[C_i].  \]
	\item For any ample $\mbR$-divisor $H\>0$
	\[\NE(X)_{K_X+\Delta\>0}+\NM(X)=\NE(X)_{K_X+\Delta+H\>0}+\sum_{i=1}^N\mbR_{\>0}[C_i].  \] 	
	\end{enumerate}
	\end{theorem}
 
{\bf Some remarks about the paper.}
	In writing this paper we have tried to give as much detail as possible even if the arguments are very similar to the characteristic $0$ case. This is for convenience, future reference, and to avoid any unpleasant surprises having to do with positive characteristic. The paper is organized in the following manner: 
1.1 and 1.2 are proved in Section 4; 1.3, 1.5 and 1.6 are in Section 5; 1.7 is in Section 6, and finally we prove 1.8 in Section 7. In Section 4 and 6 we work over a field of $\chr p>5$, in Section 5 we work over fields of arbitrary characteristic, and in Section 7 we work in $\chr 0$.\\

{\bf Acknowledgement.}  This paper originated in a conversation with Professor Burt Totaro. I would like to thank him for our fruitful discussions. I am also grateful to him for answering my questions, reading an early draft and giving valuable suggestion to improve the presentation of the paper. My sincerest gratitude goes to Professor Christopher Hacon for answering several questions, carefully reading some parts of an early draft and point out few errors. I would also like to thank Hiromu Tanaka for suggesting a quicker proof of the Theorem \ref{thm:general-bertini}. I would like to thank Joe Waldron for lot of useful discussions and the referee(s) for pointing out typos and giving valuable suggestions.\\

\section{Preliminaries}
We work with $\mbR$-Cartier divisors and use standard notations and terminologies from \cite{KM98}. We abbreviate Kawamata log terminal (resp. purely log terminal,  divisorially log terminal,  and log canonical) as KLT (resp. PLT, DLT and LC). By abuse of language we also say that $K_X+\Delta$ is KLT (resp. PLT, DLT and LC). A birational map $\phi:X\rtmap Y$ is called a \emph{birational contraction} if $\phi$ does not extract any divisor, i.e., $\phi^{-1}:Y\rtmap X$ does not contract any divisor. A projective morphism $\pi:X\to U$ is called a \emph{projective contraction} if $f_*\mcO_X=\mcO_U$. Throughout the whole paper we assume that our ground field $k$ is \textbf{algebraically closed}.\\

\begin{definition}
	Let $\phi:X\rtmap Y$ be a birational contraction of normal quasi-projective varieties. Let $D$ be a $\mbR$-Cartier divisor on $X$ such that $D'=\phi_*D$ is also $\mbR$-Cartier. We say $\phi$ is \emph{$D$-non-positive} (respectively \emph{$D$-negative}) if for some common resolution $p:W\to X$ and $q:W\to Y$, we can write 
	\[
		p^*D=q^*D'+E,
	\]
where $E\>0$ is $q$-exceptional (respectively $E\>0$ is $q$-exceptional and the support of $E$ contains the strict transform of the $\phi$-exceptional divisors).\\		
\end{definition}

\begin{definition}
	Let $\pi:X\to U$ be a projective morphism between normal quasi-projective varieties. Suppose that $K_X+\Delta$ is LC and $\phi:X\rtmap Y$ is a birational contraction of normal quasi-projective varieties over $U$ and $Y$ is projective over $U$. Set $\Gamma=\phi_*\Delta$.
	\begin{enumerate}
		\item Y is a \emph{weak log canonical model} for $K_X+\Delta$ over $U$ if $\phi$ is $(K_X+\Delta)$-non-positive and $K_Y+\Gamma$ is nef over $U$.
		\item $Y$ is a \emph{log minimal model} for $K_X+\Delta$ over $U$ if $\phi$ is $(K_X+\Delta)$-negative, $K_Y+\Gamma$ is DLT and nef over $U$, and $Y$ is $\mbQ$-factorial.\\
	\end{enumerate}	
\end{definition} 

\begin{remark}
	Note that our definition of \emph{log minimal model} is same as the \emph{log terminal model} in \cite[Definition 3.6.7]{BCHM10}.\\	
\end{remark}

\begin{definition}
Let $\pi:X\to U$ be a projective morphism of normal quasi-projective varieties and let $D$ be an $\mbR$-Cartier divisor on $X$.\\
\begin{enumerate}
	\item We say that a birational contraction $f:X\rtmap Y$ over $U$ is a \emph{semi-ample model} of $D$ over $U$ if $f$ is $D$-non-positive, $Y$ is normal and projective over $U$ and $H=f_*D$ is semi-ample over $U$.
	\item We say that a rational map $g:X\rtmap Z$ over $U$ is the \emph{ample model} for $D$ over $U$ if $Z$ is normal and projective over $U$ and there is an ample divisor $H$ over $U$ on $Z$ such that if $p:W\to X$ and $q:W\to Z$ resolve $g$, then $q$ is a contraction morphism and we may write $p^*D\sim_{\mbR, U}q^*H+E$, where $E\>0$ and for every $B\in|p^*D/U|_\mbR$, then $B\>E$.
\end{enumerate}	
\end{definition}

\begin{definition}
	Let $\pi:X\to U$ be a projective morphism between two normal quasi-projective varieties. Let $V$ be a finite dimensional affine subspace of the real vector space $\WDiv_\mbR(X)$ of Weil divisors of $X$. Fix an $\mbR$-divisor $A\>0$ and define
\begin{align*}
	V_A &=\{\Delta:\Delta=A+B, B\in V\},\\
	\mcL_A(V) &=\{\Delta=A+B\in V_A: K_X+\Delta \mbox{ is LC and } B\>0\},\\
	\mcE_{A, \pi}(V) &=\{\Delta\in\mcL_A(V): K_X+\Delta \mbox{ is pseudo-effective over } U\},\\
	\mcN_{A, \pi}(V) &=\{\Delta\in\mcL_A(V): K_X+\Delta \mbox{ is nef over } U\}.
	\end{align*}

Given a birational contraction $\phi:X\rtmap Y$ over $U$, define
\[
	\mcW_{\phi, A, \pi}(V)=\{\Delta\in\mcE_{A, \pi}(V): \phi \mbox{ is a weak log canonical model for } (X, \Delta) \mbox{ over } U\},
\]
and given a rational map $\psi:X\rtmap Z$ over $U$, define
\[
	\mcA_{\psi, A, \pi}(V)=\{\Delta\in\mcE_{A, \pi}(V): \psi \mbox{ is the ample model for } (X, \Delta) \mbox{ over } U\}\\.
\]
\end{definition}



\section{Tanaka's Bertini type theorem}
Following is a generalization of Tanaka's result on Bertini-type theorems in positive characteristic \cite[Theorem 1]{Tan17}. The proof presented here is suggested by Tanaka.
\begin{theorem}\label{thm:general-bertini}
	Fix $\mbK\in\{\mbQ, \mbR\}$. Let $X$ be a projective variety over a field $k$ containing an infinite perfect subfield $k_0$ of characteristic $p>0$. We assume that log resolution exists. Let $D$ be a semi-ample $\mbK$-Cartier $\mbK$-divisor on $X$ and $Z_1, Z_2, \ldots, Z_l$ are finitely many closed subsets of $X$. Then the following hold:
	\begin{enumerate}
		\item If $(X, \Delta_1\>0), (X, \Delta_2\>0),\ldots, (X, \Delta_m\>0)$ are all KLT pairs, then there exists an effective $\mbK$-Carter $\mbK$-divisor $0\<D'\sim_\mbK D$ not containing any $Z_j, 1\<j\<l$ in its support such that $(X, \Delta_1+D'), (X, \Delta_2+D'),\ldots, (X, \Delta_m+D')$ are all KLT.\\
		\item If $(X, \Delta_1\>0), (X, \Delta_2\>0),\ldots, (X, \Delta_m\>0)$ are all LC pairs, then there exists an effective $\mbK$-Carter $\mbK$-divisor $0\<D'\sim_\mbK D$ not containing any $Z_j, 1\<j\<l$ in its support such that $(X, \Delta_1+D'), (X, \Delta_2+D'),\ldots, (X, \Delta_m+D')$ are all LC.
		\end{enumerate} 
\end{theorem}

\begin{proof}
	Let $f:Y\to X$ be a log resolution of the pairs $(X, \Delta_i\>0)$ for all $i=1, 2,\ldots, m$ and 
	\begin{equation}\label{eqn:tanaka-bertini}
		K_Y+\Gamma_i=f^*(K_X+\Delta_i)+F_i,
	\end{equation}
where $\Gamma_i\>0$ and $F_i\>0$ do not share any common component and $f_*\Gamma_i=\Delta_i$ and $f_*F_i=0$, for all $i=1, 2,\ldots, m$.\\

 First we will deal with the log canonical case. Define a divisor $\Delta_Y$ on $Y$ as $\Delta_Y:=(f^{-1}_*(\Delta_1+\Delta_2+\cdots+\Delta_m))_{\rm{red}}+E$, where $E$ is the reduced $f$-exceptional divisor with $\Supp E=\Ex(f)$. Then $(Y, \Delta_Y\>0)$ is a LC pair. By \cite[Theorem 1]{Tan17} and its proof it follows that there exists an effective $\mbK$-divisor $0\<D'\sim_\mbK D$ not containing any $Z_j, 1\<j\<l$ in its support such that $(Y, \Delta_Y+f^*D')$ is LC. From \eqref{eqn:tanaka-bertini} we see that $\Gamma_i+f^*D'\<\Delta_Y$ and thus $(Y, \Gamma_i+f^*D')$ is LC, for all $i=1, 2,\ldots, m$. It then follows that $(X, \Delta_i+D')$ is LC for all $i=1, 2, \ldots, m$.\\

We will now work with the KLT case. The divisor $mD$ is semi-ample $\mbK$-Cartier $\mbK$-divisor on $X$. Then by the log canonical case there exists an effective $\mbK$-divisor $0\<D''\sim_\mbK mD$ not containing any $Z_j, 1\<j\<l$ in its support such that $(X, \Delta_i+D'')$ is LC for all $i=1, 2, \ldots, m$. Set $D':=\frac{1}{m}D''\sim_\mbK D$. We claim that $(X, \Delta_i+D')$ is KLT for all $i=1, 2,\ldots, m$. Since $(X, \Delta_i)$ is KLT, if $(X, \Delta_i+D')$ is not KLT then this will implies that $(X, \Delta_i+D'')$ is not LC (for $m\>2$), a contradiction. Therefore $(X, \Delta_i+D')$ is KLT for all $i=1, 2,\ldots, m$.\\

	\end{proof}

\section{Finiteness of log minimal models}
In this section we will prove the finiteness of log minimal models, namely Theorem \ref{thm:finite-mmp}. The ideas and techniques used here are based on the paper \cite{BCHM10}.\\

\begin{lemma}\label{lem:semi-ample-to-ample}
	Let $\pi:X\to U$ be a projective morphism of normal projective varieties with $\dim X=3$ and $\chr p>5$. Suppose that $(X, \Delta)$ is a KLT pair, where $\Delta$ is big over $U$.\\
	If $\phi:X\rtmap Y$ is a weak log canonical model of $K_X+\Delta$ over $U$, then
	\begin{enumerate}
		\item $\phi$ is a semi-ample model over $U$.
		\item the ample model $\psi:X\rtmap Z$ of $K_X+\Delta$ over $U$ exists, and
		\item there is a projective contraction $h:Y\to Z$ such that $K_Y+\Gamma\sim_{\mbR, U}h^*H$, for some ample $\mbR$-divisor $H$ over $U$, where $\Gamma=\phi_*\Delta$.
		\end{enumerate}
	\end{lemma}

\begin{proof}
	Since $K_Y+\Gamma$ is KLT and $\Gamma$ is big, it follows from \cite[Theorem 1.2]{BW17} that $K_Y+\Gamma$ is semi-ample. Part $(2)$ and $(3)$ then follow from \cite[Lemma 3.6.6(3)]{BCHM10}.\\
	
	\end{proof}


	

\begin{lemma}\label{lem:dlt-to-Q-factorization}
	Let $(X, \Delta\>0)$ be a $3$-fold projective DLT pair in $\chr p>5$. Then there exists a small birational morphism $\pi:Y\to X$ from a $\mbQ$-factorial normal projective $3$-fold $Y$ such that \[
		K_Y+\Delta_Y=\pi^*(K_X+\Delta),
	\]
and $(Y, \Delta_Y\>0)$ is DLT.		   
\end{lemma}

\begin{proof}
	Since $(X, \Delta)$ is DLT, there exists a log resolution $f:X'\to X$ of $(X, \Delta)$ such that $a(E_i, X, \Delta)>-1$ for every $f$-exceptional divisor $E_i$. We can write
	\[
		K_{X'}+\Gamma=f^*(K_X+\Delta)+E,
	\]
where $\Gamma\>0$ and $E\>0$ do not share any common component, and $f_*\Gamma=\Delta, f_*E=0$.\\

Let $F\>0$ be a reduced divisor with $\Supp F=\Ex(f)$. Then $(X', \Gamma+\ve F)$ is DLT for $0<\ve\ll 1$. Furthermore, we have
\[
	K_{X'}+\Gamma+\ve F\sim_{\mbR, f} E+\ve F\>0.
\]	
By \cite[Corollary 1.8]{Wal17} we can run a terminating $(K_{X'}+\Gamma+\ve F)$-MMP over $X$. Let $\phi:X'\rtmap Y$ be the corresponding minimal model over $X$, i.e., $K_Y+\Delta_Y+\ve\phi_*F$ is nef over $X$, where $K_Y+\Delta_Y+\ve\phi_*F=\phi_*(K_{X'}+\Gamma+\ve F)$. It is easy to see from the Negativity Lemma that $\phi$ contracts all $f$-exceptional divisors, i.e., $\phi_*F=0$; in particular, if $\pi:Y\to X$ is the structure morphism, then $\pi$ is a small birational morphism, $Y$ is $\mbQ$-factorial, $(Y, \Delta_Y\>0)$ is DLT and 
\[
	K_Y+\Delta_Y=\pi^*(K_X+\Delta).
\]

\end{proof}

\begin{remark}\label{rmk:bchm-lemmas}
	In the proof of the following results we will use Lemma 3.6.12, 3.7.2, 3.7.3, 3.7.4 and 3.7.5 from \cite{BCHM10}. Proofs of these lemmas depend on the Bertini's theorem for base-point free linear system in characteristic $0$. However, their proofs use Bertini's theorem in one specific way, namely, given an ample$/U$ $\mbQ$-divisor $A\>0$ and finitely many LC pairs $(X, \Delta_i\>0), i=1, 2,\ldots, m$, there exists an effective divisor $0\<A'\sim_{\mbQ, U}A$ such that $(X, \Delta_i+A')$ is LC for all $i=1, 2,\ldots, m$. We note that in the following results in characteristic $p>0$ our set up is: $\pi:X\to U$ is a projective morphism and $X$ and $U$ are both projective varieties. So for an ample$/U$ $\mbQ$-divisor $A\>0$ there exists an effective divisor $0\<A'=A+l\pi^*H\sim_{\mbQ, U} A, l\gg 0$ such that $(X, \Delta_i+A')$ is LC for all $i$ by Theorem \ref{thm:general-bertini}, where $H$ is an ample divisor on $U$. In particular, the proofs of those lemmas from \cite{BCHM10} hold in our settings. In the following Lemma \ref{lem:convex-lc-modification} we give a sketch of the proof of \cite[Lemma 3.7.4]{BCHM10} in our settings explaining the use of Theorem \ref{thm:general-bertini} in place of Bertini's theorem.\\
	\end{remark}

	\begin{lemma}\label{lem:convex-lc-modification}
		Let $\pi:X\to U$ be a projective morphism between two normal projective varieties. Assume that log resolution exists. Let $V$ be a finite dimensional affine subspace of $\WDiv_\mbR(X)$, which is defined over the rationals,  and let $A$ be a general ample $\mbQ$-divisor over $U$. Let $S$ be a sum of prime divisors. Suppose that there is a DLT pair $(X, \Delta_0)$, where $S=\lrd\Delta_0\rrd$, and let $G\>0$ be any divisor whose support does not contain any LC centers of $(X, \Delta_0)$.\\
		Then we may find a general ample $\mbQ$-divisor $A'\>0$ over $U$, an affine subspace $V'$ of $\WDiv_\mbR(X)$, which is defined over the rationals, and a rational affine linear isomorphism
	\[
		L:V_{S+A}\to V'_{S+A'}
	\]	
	so that
	\begin{enumerate}
		\item $L$ preserves $\mbQ$-linear equivalence over $U$.
		\item $L(\mcL_{S+A}(V))$ is contained in the interior of $\mcL_{S+A'}(V')$.
		\item For any $\Delta\in L(\mcL_{S+A}(V))$, $K_X+\Delta$ is DLT and $\lrd\Delta\rrd=S$, and
		\item For any $\Delta\in L(\mcL_{S+A}(V))$, the support of $\Delta$ contains the support of $G$.\\
	\end{enumerate}	
	\end{lemma}
	
\begin{proof}[Sketch of the Proof]
We will only explain the part where Bertini's theorem is used in the proof of \cite[Lemma 3.7.4]{BCHM10}, which is basically the second paragraph in \cite[Page 436]{BCHM10}. All other arguments in the rest of the proof of \cite[Lemma 3.7.4]{BCHM10} holds in our settings here without any change.\\
We will basically show that we can choose ample divisors $A_i$ and $A'$ as in the proof of \cite[Lemma 3.7.4]{BCHM10} such that $(X, \Delta+A'-A)$ is LC, $(X, \Delta+4/3A_i+A'-A)$ is LC for all $1\<i\<l$ and for all $\Delta\in\mcL_{S+A}(V)$, and $(X, A'+\Delta_0)$ is DLT.\\
 To that end, let $\Gamma_1, \Gamma_2,\ldots, \Gamma_m$ be the vertices of the rational polytope $\mcL_{S+A}(V)$. Since $(X, \Gamma_i)$ is LC for all $1\<j\<m$ and $(X, \Delta_0)$ is DLT, by Theorem \ref{thm:general-bertini} there exists a divisor $0\<A''\sim_{\mbQ}A$ such that $(X,\Gamma_j+A'')$ is LC for all $1\<j\<m$ and the support of $A''$ does not contain any LC center of $(X, \Delta_0)$. Set $A'=\ve A''$ for a rational number $\ve\in(0, 1/4]$. For $0<\ve\ll 1/4$ we see that $(X, \Gamma_j+A')$ is LC for all $1\<j\<m$ and $(X, \Delta_0+A')$ is DLT. Furthermore, since $A_i$'s are general ample $\mbQ$-divisors and $0<\ve\ll 1/4$, it again follows from Theorem \ref{thm:general-bertini} that $(X, \Gamma_j+4/3A_i+A')$ is LC for all $1\<j\<m$ and $1\<i\<l$. Now for any $\Delta\in\mcL_{S+A}(V)$ we can write $\Delta=\sum_{j=1}^m\lambda_j\Gamma_j$ for some $\lambda_j\>0$ such that $\sum_{j=1}^m\lambda_j=1$. It is easy to see that convex sum of finitely many LC divisors are LC. It then follows that $(X, \Delta+A')$ and $(X, \Delta+4/3A_i+A')$ are both LC for all $1\<i\<l$ and for all $\Delta\in\mcL_{S+A}(V)$. In particular, we finally have $(X, \Delta+A'-A)$ is LC, $(X, \Delta+4/3A_i+A'-A)$ is LC for all $1\<i\<l$ and for all $\Delta\in\mcL_{S+A}(V)$, and $(X, A'+\Delta_0)$ is DLT.\\

\end{proof}

\begin{definition}
	Given an extremal ray $R\subset\NE(X)$, we define a hyperplane 
	\[ R^\bot=\{\Delta\in\mcL(V): (K_X+\Delta)\cdot R=0\}. \]
	\end{definition}

\begin{theorem}\label{thm:finite-hyperplane}
Let $\pi:X\to U$ be a projective contraction between two normal projective varieties with $\dim X=3$ and $\chr p>5$. Let $V$ be a finite dimensional affine subspace of $\WDiv_\mbR(X)$ which is defined over the rationals. Fix a genral ample$/U$ $\mbQ$-divisor $A$ on $X$. Suppose that there is a divisor $\Delta_0\>0$ such that $K_X+\Delta_0$ is KLT.\\
Then the set of hyperplanes $R^\bot$ is finite in $\mcL_A(V)$, as $R$ ranges over the set of all extremal rays of $\NE(X/U)$. In particular, $\mcN_{A,\pi}(V)$ is a rational polytope.\\
	\end{theorem}

\begin{proof}
	This result corresponds to Theorem 3.11.1 in \cite{BCHM10}.\\
	Since $\mcL_A(V)$ is compact, it is enough to prove the finiteness of $R^\bot$ locally in a neighborhood of a point $\Delta\in\mcL_A(V)$. Now since there is a boundary divisor $\Delta_0$ such that $(X, \Delta_0)$ is KLT and the image of a hyperplane under linear isomorphism of affine spaces is again a hyperplane, by Lemma \ref{lem:convex-lc-modification} we may assume that $K_X+\Delta$ is KLT. Fix $\ve>0$ such that if $\Delta'\in\mcL_A(V)$ and $||\Delta'-\Delta||<\ve$, then $\Delta'-\Delta+A/2$ is ample over $U$. Let $R$ be an extremal ray over $U$ such that $(K_X+\Delta')\cdot R=0$ for some $\Delta'\in\mcL_A(V)$ with $||\Delta'-\Delta||<\ve$. Then we have
	\[(K_X+\Delta-A/2)\cdot R=(K_X+\Delta')\cdot R-(\Delta'-\Delta+A/2)\cdot R=-(\Delta'-\Delta+A/2)\cdot R<0.\]
Write $\Delta=A+B$. Then $K_X+\Delta-A/2=K_X+B+A/2$. By \cite[Theorem 1.7(3)]{Wal17} there are only finitely many extremal rays $R$ satisfying these properties.\\

Now $\mcN_{A,\pi}(V)$ is clearly a closed subset of $\mcL_A(V)$. Let $\Delta\in\mcL_A(V)$. If $K_X+\Delta$ is not nef$/U$, then again by \cite[Theorem 1.7]{Wal17} there exists an extremal ray $R$ of $\NE(X/U)$ generated by a rational curve $\Sigma$ such that $(K_X+\Delta)\cdot\Sigma<0$. In particular, $\mcN_{A, \pi}(V)$ is contained in the half-spaces $R^{\>0}=\{\Gamma\in\mcL_A(V):(K_X+\Gamma)\cdot R\>0\}$ of the hyperplanes $R^\bot$. Then by the previous part, there exists finitely many extremal rays $R_1, R_2,\ldots, R_n$ of $\NE(X/U)$ such that $\mcN_{A,\pi}(V)=\cap_{i=1}^nR^{\>0}_i$. Since $R_i$'s are generated by irreducible curves, the hyperplanes $R^\bot_i$'s are all rational hyperplanes, in particular, $\mcN_A(V)$ is a rational polytope.\\ 
	
	\end{proof}

\begin{corollary}\label{cor:weak-polytope}
	Let $\pi:X\to U$ be a projective contraction between two normal projective varieties with $\dim X=3$ and $\chr p>5$. Let $V$ be a finite dimensional affine subspace of $\WDiv_\mbR(X)$ which is defined over the rationals. Fix a general ample$/U$ $\mbQ$-divisor $A$ on $X$. Suppose that there is a divisor $\Delta_0\>0$ such that $K_X+\Delta_0$ is KLT. Let $\phi:X\rtmap Y$ be a birational contraction over $U$.\\	
	Then $\mcW_{\phi, A, \pi}(V)$ is a rational polytope. Moreover, there are finitely many morphisms $f_i:Y\to Z_i$ over $U$, $1\<i\<k$, such that if $f:Y\to Z$ is any contraction over $U$ and there is an $\mbR$-divisor $D$ on $Z$, which is ample over $U$, such that $K_Y+\Gamma=\phi_*(K_X+\Delta)\sim_{\mbR, U}f^*D$ for some $\Delta\in\mcW_{\phi, A, \pi}(V)$, then there is an index $1\<i\<k$ and an isomorphism $\eta:Z_i\to Z$ such that $f=\eta\circ f_i$.
	\end{corollary}

\begin{proof}
	This result corresponds to Corollary 3.11.2 in \cite{BCHM10}.\\
	Replacing $V_A$ by the span of $\mcL_A(V)$ if necessary we may assume that $\mcL_A(V)$ spans $V_A$. By compactness, to prove that $\mcW_{\phi,A,\pi}(V)$ is a rational polytope, we may work locally about a divisor $\Delta\in\mcW_{\phi,A,\pi}(V)$. By \cite[Lemma 3.7.4]{BCHM10} we may assume that $K_X+\Delta$ is KLT. Then $K_Y+\Gamma=\phi_*(K_X+\Delta)$ is KLT as well. Let $C=\phi_*A$ and $W=\phi_*(V)$. Then $C$ is a big$/U$ $\mbQ$-divisor on $Y$. By \cite[Lemma 3.7.3 and 3.7.4]{BCHM10} there exists a rational affine linear isomorphism $L:W\to W'$ and an ample$/U$ $\mbQ$-divisor $C'$ such that $L(\Gamma)$ is contained in the interior of $\mcL_{C'}(W')$ and $L(\Psi)\sim_{\mbQ, U}\Psi$ for every $\Psi\in W$. Then by Theorem \ref{thm:finite-hyperplane}, $\mcN_{C',\psi}(W')$ is a non-empty rational polytope containing $L(\Gamma)$, where $\psi:Y\to U$ is the structure morphism. Therefore $\mcN_{C,\psi}(W)$ is a rational polytope locally around $\Gamma=\phi_*\Delta$.\\
	
	Consider the following resolution of $\phi:X\rtmap Y$ which is also a log resolution of $(X, \Delta)$.
	\[\xymatrixcolsep{3pc}\xymatrixrowsep{3pc}\xymatrix{ & W\ar[ld]_p\ar[rd]^q &\\
		 												X\ar@{-->}[rr]^\phi && Y.	} \]	
Then we have 
\[\begin{split}
K_W+\Psi &=p^*(K_X+\Delta)\\
K_W+\Phi &=q^*(K_Y+\Gamma).
\end{split} \]	
	
Note that $\Delta\in\mcW_{\phi,A,\pi}(V)$ if and only if $\Gamma=\phi_*\Delta\in\mcN_{C,\psi}(W)$ and $\Psi-\Phi\>0$. Since the map $L:V\to W$ given by $\Delta\to \Gamma=\phi_*\Delta$ is rational and linear, in a neighborhood of $\Delta$, $\mcW_{\phi,A,\pi}(V)$ is cut out from $\mcL_A(V)$ by finitely many half-spaces generated by affine rational hyperplanes. Therefore by compactness, $\mcW_{\phi,A,\pi}(V)$ is a rational polytope.\\


Now for any $\Delta\in \mcW_{\phi,A,\pi}(V)$ we have $\Gamma=\phi_*\Delta=A'+B'$, where $A'\>0$ is a big divisor on $Y$. Therefore by perturbing $\Gamma$, from the base-point free theorem \cite[Theorem 1.2]{BW17} it follows that there is a contraction $g:Y\to Z$ satisfying the required conditions. The rational map $g\circ\phi:X\rtmap Z$ is the ample model of $K_X+\Delta$. Next we prove that there are only finitely many such contractions $g:Y\to Z$ corresponding to all $\Delta\in \mcW_{\phi,A,\pi}(V)$.\\
For two contractions $f:Y\to Z, f_*\mcO_Y=\mcO_Z$ and $f':Y\to Z', f'_*\mcO_Y=\mcO_{Z'}$ over $U$, there exists an isomorphism $\eta:Z\to Z'$ satisfying $f'=\eta\circ f$ if and only if $f$ and $f'$ contracts exactly same curves, i.e., $f(C)=\pt$ if and only if $f'(C)=\pt$ for irreducible curves $C\subset X$ (see \cite[Pro. 1.14 and Lem. 1.15]{Deb01}).
Let $f:Y\to Z$ be a contraction over $U$ such that 
\[K_Y+\Gamma=K_Y+\phi_*\Delta\sim_{\mbR, U}f^*D, \]	
	where $\Delta\in\mcW_{\phi,A,\pi}(V)$ and $D$ is an ample$/U$ $\mbR$-divisor on $Z$. $\Gamma$ belongs to the interior of a unique face $G$ of $\mcN_{C,\psi}(W)$. Let $\Gamma_1, \Gamma_2,\ldots, \Gamma_k$ be the vertices of the $G$. Write $\Gamma=\sum^k_{j=1}\lambda_j\Gamma_j$, where $\sum^k_{j=1}\lambda_j=1$ and $\lambda_j\>0$ for all $j=1, 2,\ldots, k$. Since $\Gamma$ is contained in the interior of $G$, for any given (fixed) index $i$, we can choose $\lambda_i>0$ in $\Gamma=\sum^k_{j=1}\lambda_j\Gamma_j$. Let $C\subset Y$ be a curve contracted by $f$, then $0=(K_Y+\Gamma)\cdot C=\sum^k_{j=1}\lambda_j(K_Y+\Gamma_j)\cdot C\>0$. This implies that $\lambda_i(K_Y+\Gamma_i)\cdot C=0$, i.e., $(K_Y+\Gamma_i)\cdot C=0$, since $\lambda_i\neq 0$. Therefore if $C$ is contracted by $f$, then $(K_Y+\Gamma_j)\cdot C=0$ for all $j=1, 2,\ldots, k$. Conversely, if $C$ is a curve on $Y$ such that $(K_Y+\Gamma_j)\cdot C=0$ for all $j=1, 2,\ldots, k$, then clearly $(K_Y+\Gamma)\cdot C=0$, and hence $C$ is contracted by $f$. Therefore the curves contacted by $f$ are uniquely determined by $G$. Now $\Delta$ is contaied in the interior of a unique face $F$ of $\mcW_{\phi,A,\pi}(V)$ and $G$ is determined by $F$. But since $\mcW_{\phi,A,\pi}(V)$ is a rational polytope it has only finitely many faces.\\

	\end{proof}

\begin{corollary}\label{cor:local-LTM}
Let $\pi:X\to U$ be a projective contraction between two normal projective varieties with $\dim X=3$ and $\chr p>5$. Let $V$ be a finite dimensional affine subspace of $\WDiv_\mbR(X)$ which is defined over the rationals. Fix a general ample$/U$ $\mbQ$-divisor $A\>0$ on $X$. Suppose that there is a divisor $\Delta_0\>0$ such that $K_X+\Delta_0$ is KLT. Let $f:X\to Z$ be a morphism over $U$ such that $\Delta_0\in\mcL_A(V)$ and $K_X+\Delta_0\sim_{\mbR, U}f^*H$, where $H$ is an ample divisor over $U$. Let $\phi:X\rtmap Y$ be a birational contraction over $Z$.\\
Then there is a neighborhood $P_0$ of $\Delta_0$ in $\mcL_A(\Delta)$ such that for all $\Delta\in P_0$, $\phi$ is a log minimal model for $K_X+\Delta$ over $Z$ if and only if $\phi$ is a log minimal model for $K_X+\Delta$ over $U$. 	
	\end{corollary}

\begin{proof}
	This result corresponds to Corollary 3.11.3 in \cite{BCHM10}.\\
	By Theorem \ref{thm:finite-hyperplane} there exists finitely many extremal rays $R_1, R_2,\ldots, R_k$ of $\NE(Y/U)$ such that if $K_Y+\Gamma=K_Y+\phi_*\Delta$ is not nef over $U$ for some $\Delta\in\mcL_{A}(V)$, then it is negative on one of these rays. If $\Gamma_0=\phi_*\Delta_0$, then we may write 
	\[K_Y+\Gamma=K_Y+\Gamma_0+(\Gamma-\Gamma_0)\sim_{\mbR, U}g^*H+\phi_*(\Delta-\Delta_0),\] 
	where $g:Y\to Z$ is the structure morphism.\\
	
\textbf{Claim:} If $\Delta\in\mcL_A(V)$ is sufficiently close to $\Delta_0$ and $(K_Y+\Gamma)\cdot R_{i_0}=\phi_*(K_X+\Delta)\cdot R_{i_0}<0$ for some $i_0\in\{1, 2,\ldots, k\}$, then $(K_Y+\Gamma_0)\cdot R_{i_0}=\phi_*(K_X+\Delta_0)\cdot R_{i_0}=0$. 
\begin{proof}[Proof of the Claim]
	On the contrary assume that $(K_Y+\Gamma_0)\cdot R_{i_0}>0$. Let $\alpha=\frac{1}{2}(K_Y+\Gamma_0)\cdot R_{i_0}>0$. Then $(K_Y+\Gamma_0)\cdot R_{i_0}>\alpha$. Let $\Gamma'=\phi_*\Delta'$ for some $\Delta'\in\mcL_{A}(V)$ such that $K_Y+\Gamma'$ is LC. Then by \cite[Theorem 1.7]{Wal17} we have $(K_Y+\Gamma')\cdot R_{i_0}\>-6$. Choose $\Delta\in\mcL_A(V)$ such that $\Gamma=\phi_*\Delta$ lies on the line segment joining $\Gamma_0$ and $\Gamma'$, i.e., $\Gamma=r\Gamma_0+s\Gamma'$ for some $r\>0$ and $s>0$ satisfying $r+s=1$. Then
	\[(K_Y+\Gamma)\cdot R_{i_0}=r(K_Y+\Gamma_0)\cdot R_{i_0}+s(K_Y+\Gamma')\cdot R_{i_0}>2r\alpha-6s>0 \quad\mbox{if } r>\frac{3s}{\alpha}. \]
	This is a contradiction. Therefore if $\Delta$ is sufficiently close to $\Delta_0$ then $(K_Y+\Gamma)\cdot R_{i_0}<0$ implies that $(K_Y+\Gamma_0)\cdot R_{i_0}=0$.\\
	\end{proof} 
	In other words, there exists a neighborhood $P_0$ of $\Delta_0$ in $\mcL_A(V)$ such that if $\Delta\in P_0$ and $K_Y+\Gamma=\phi_*(K_X+\Delta)$ is not nef over $U$, then $(K_Y+\Gamma)\cdot R_{i_0}<0$ for some $i_0\in\{1, 2,\ldots, k\}$ and $R_{i_0}$ is extremal over $Z$ (otherwise $(K_Y+\Gamma_0)\cdot R_{i_0}>0$), and consequently $K_Y+\Gamma$ is not nef over $Z$. Contra-positively, if $\phi$ is a log minimal model of $K_X+\Delta$ over $Z$, then it is a log minimal model over $U$. The other direction is obvious.\\

	\end{proof}

\begin{proposition}\label{pro:polytope-finite}
Let $\pi:X\to U$ be a projective contraction between two normal projective varieties with $\dim X=3$ and $\chr p>5$. Let $V$ be a finite dimensional affine subspace of $\WDiv_\mbR(X)$ which is defined over the rationals. Fix a general ample$/U$ $\mbQ$-divisor $A$ on $X$. Let $\mcC\subset\mcL_A(V)$ be a rational polytope such that if $\Delta\in\mcC$, then $K_X+\Delta$ is KLT.\\
Then there are finitely many rational maps $\phi_i:X\rtmap X_i$ over $U$, $1\<i\<k$, with the property that if $\Delta\in\mcC\cap\mcE_{A,\pi}(V)$, then there is an index $1\<j\<k$ such that $\phi_j$ is a log minimal model of $K_X+\Delta$ over $U$.\\
	\end{proposition}

\begin{remark}
	This proposition is proved in \cite[Theorem 1.4]{BW17} with the additional hypothesis that $X$ is $\mbQ$-factorial. One can conceivably prove the above statement using \cite[Theorem 1.4]{BW17} by going to a $\mbQ$-factorization of $X$. However, we take a different approach here, we use the techniques of \cite[Lemma 7.1]{BCHM10} which fits better with rest of the paper.\\
\end{remark}

\begin{proof}[Proof of Proposition \ref{pro:polytope-finite}]
	Replacing $V_A$ by the span of $\mcC$ if necessary we may assume that $\mcC$ spans $V_A$. We proceed by induction on the dimension of $\mcC$.\\
First assume that $\dim \mcC=0$. Then $\mcC=\{\Delta_0\}$ for some $\Delta_0\in\mcL_A(V)$. If $\Delta_0\in\mcE_{A, \pi}(V)$, then by \cite[Theorem 1.2]{Bir16} there exists a log minimal model $\phi:X\rtmap Y/U$ for $K_X+\Delta_0$. By induction assume that the statement is true for any such rational polytope $\mcC'$ with $\dim\mcC'<\dim\mcC$.\\

Now we will prove the statement assuming that there is a divisor $\Delta_0\in\mcC$ such that $K_X+\Delta_0\sim_{\mbR, U} 0$. Let $\Delta\in\mcC$ be a divisor such that $\Delta\neq\Delta_0$. Then there exists a divisor $\Delta'$ on one of the faces of $\mcC$ such that 
\[\Delta=\lambda\Delta'+(1-\lambda)\Delta_0, \] 
for some $0<\lambda\<1$.\\

We have 
\[K_X+\Delta=\lambda(K_X+\Delta')+(1-\lambda)(K_X+\Delta_0)\sim_{\mbR, U}\lambda(K_X+\Delta').\]
Therefore $\Delta\in\mcE_{A,\pi}(V)$ if and only if $\Delta'\in\mcE_{A,\pi}(V)$, and by \cite[Lemma 3.6.9]{BCHM10} $K_X+\Delta$ and $K_X+\Delta'$ have same log minimal models over $U$. Since $\mcC$ is a rational polytope, it has finitely many faces each of which are rational polytope themselves, therefore by induction we are done.\\

Now we will prove the general case. Applying \cite[Lemma 3.7.4]{BCHM10} we may assume that $\mcC$ is contained in the interior of $\mcL_A(V)$. Note that $\mcC\cap\mcE_{A,\pi}(V)$ is compact (as $\mcL_A(V)$ is compact and $\mcC\cap\mcE_{A, \pi}$ is closed). So it is sufficient to prove the statement locally in a neighborhood of a divisor $\Delta_0\in\mcC\cap\mcE_{A,\pi}(V)$.\\
Let $\phi:X\rtmap Y/U$ be a log minimal model for $K_X+\Delta_0$ and $\Gamma_0=\phi_*\Delta_0$. Let $\mcC_0\subset\mcL_A(V)$ be a neighborhood around $\Delta$, which is also a rational polytope. Since $\phi$ is $(K_X+\Delta_0)$-negative, by shirking $\mcC_0$ (without changing its dimension) around $\Delta_0$ we may assume that $a(F, K_X+\Delta)<a(F, K_Y+\Delta)$ for all $\Delta\in\mcC_0$ and for all $\phi$-exceptional divisors $F$. Note that $K_Y+\Gamma_0$ is KLT and $Y$ is $\mbQ$-factorial. Since KLT is an open condition, all nearby divisors of $\Gamma_0$ in $Y$ are also KLT. Therefore by shrinking $\mcC_0$ further around $\Delta_0$ we may assume that $K_Y+\Gamma$ is KLT for all $\Delta\in\mcC_0$, where $\Gamma=\phi_*\Delta$.\\
Replacing $\mcC$ by $\mcC_0$ we may assume that the rational polytope $\mcC'=\phi_*(\mcC)$ is contained in $\mcL_{\phi_*A}(W)$, where $W=\phi_*V$. Note that $\phi_*A$ is not an ample divisor, however it is a big $\mbQ$-divisor on $Y$. By \cite[Lemma 3.7.3]{BCHM10} there exists a rational affine linear isomorphism $L:W\to V'$ on $Y$ and a general ample$/U$ $\mbQ$-divisor $A'$ on $Y$ such that $L(\mcC')\subset\mcL_{A'}(V')$, $L(\Gamma)\sim_{\mbQ, U}\Gamma$ for all $\Gamma\in\mcC'$ and $K_Y+\Gamma$ is KLT for any $\Gamma\in L(\mcC')$.\\
By \cite[Lemma 3.6.9 and 3.6.10]{BCHM10}, any log minimal model of $(Y, L(\Gamma))$ over $U$ is also a log minimal model of $(X, \Delta)$ over $U$ for every $\Delta\in\mcC$. Thus replacing $X$ by $Y$ and $\mcC$ by $L(\mcC')$ we may assume that $X$ is $\mbQ$-factorial and $K_X+\Delta_0$ is $\pi$-nef. Since $\Delta_0$ is a big divisor, by the base-point free theorem \cite[Theorem 1.2]{BW17} $K_X+\Delta_0$ has an ample model $\psi:X\to Z$. In particular, $K_X+\Delta_0\sim_{\mbR, Z}0$.
By the case we have already proved, there exist finitely many birational maps $\phi_i:X\rtmap Y_i$ over $Z$, $1\<i\<k$, such that for any $\Delta\in\mcC\cap\mcE_{A,\psi}(V)$, there is an index $i$ such that $\phi_i$ is a log minimal model of $K_X+\Delta$ over $Z$. Since there are only finitely many indices $1\<i\<k$, by shrinking $\mcC$ (without changing its dimension) if necessary, it follows from Corollary \ref{cor:local-LTM} that if $\Delta\in\mcC$, then $\phi_i$ is a log minimal model for $K_X+\Delta$ over $Z$ if and only if it is a log minimal model for $K_X+\Delta$ over $U$.\\
 Let $\Delta\in\mcC\cap\mcE_{A,\pi}(V)$. Then $\Delta\in\mcC\cap\mcE_{A,\psi}(V)$, and there exists an index $1\<j\<k$ such that $\phi_j$ is a log minimal model for $K_X+\Delta$ over $Z$. But then $\phi_j$ is a log minimal model for $K_X+\Delta$ over $U$.\\

	\end{proof}

\begin{theorem}\label{thm:polytope-weak-mmp}
	Let $\pi:X\to U$ be a projective contraction between two normal projective varieties with $\dim X=3$ and $\chr p>5$. Suppose that there is a KLT pair $(X, \Delta_0\>0)$. Fix $A\>0$, a general ample$/U$ $\mbQ$-divisor. Let $V$ be a finite dimensional affine subspace of $\WDiv_\mbR(X)$ which is defined over the rationals. Let $\mcC\subset\mcL_A(V)$ be a rational polytope.\\
	
	 Then there are finitely many birational maps $\psi_j:X\rtmap Z_j$ over $U$, $1\<j\<l$ such that if $\psi:X\rtmap Z$ is a weak log canonical model of $K_X+\Delta$ over $U$, for some $\Delta\in\mcC$, then there is an index $1\<j\<l$ and an isomorphism $\xi:Z_j\to Z$ such that $\psi=\xi\circ\psi_j$.
	\end{theorem}

\begin{proof}
	This result corresponds to Lemma 7.2 in \cite{BCHM10}.\\
	By \cite[Lemma 3.7.4]{BCHM10} for every $\Delta\in\mcC$ there exists a $\Delta'\>0$ such that $K_X+\Delta'$ is KLT and $K_X+\Delta'\sim_{\mbR, U}K_X+\Delta$. Then by \cite[Lemma 3.6.9]{BCHM10} $\psi:X\rtmap Z$ is a weak log canonical model of $K_X+\Delta$ over $U$ if and only if $\psi'$ is a weak log canonical model of $K_X+\Delta'$ of over $U$. Therefore by \cite[Lemma 3.7.4]{BCHM10} we may assume that $K_X+\Delta$ is KLT for every $\Delta\in\mcC$.\\

Let $G\>0$ be a divisor such that it contains the support of every divisor in $V$ and $f:Y\to X$ a log resolution of $(X, G)$. For a $\Delta\in\mcL_A(V)$ we can write
\[K_Y+\Gamma=f^*(K_X+\Delta)+E,    \]	
where $\Gamma\>0$ and $E\>0$ have no common components, $f_*\Gamma=\Delta$ and $f_*E=0$.\\
If $\psi:X\rtmap Z$ is a weak log canonical model of $K_X+\Delta$ over $U$, then $\psi\circ f:Y\rtmap Z$ is a weak log canonical model of $K_Y+\Gamma$ over $U$. Let $\mcC'$ be the image of $\mcC$ under the map $\Delta\to\Gamma$. Then $\mcC'$ is a rational polytope and $K_Y+\Gamma$ is KLT for all $\Gamma\in\mcC'$. In particular, $\mathbf{B}_+(f^*A/U)$ does not contain any LC centers of $K_Y+\Gamma$ for any $\Gamma\in\mcC'$. Let $W$ be the subspace of $\WDiv_\mbR(Y)$ spanned by the strict transforms of the components of $G$ and the exceptional divisors of $f$. Then by \cite[Lemma 3.7.3 and 3.6.9]{BCHM10} we may assume that there exists a general ample$/U$ $\mbQ$-divisor $A'$ on $Y$ such that $\mcC'\subset\mcL_{A'}(W)$. Replacing $X$ by $Y$ and $\mcC$ by $\mcC'$ we assume that $X$ is smooth.\\

Let $H_1\>0, H_2\>0,\ldots, H_q\>0$ be general ample$/U$ $\mbQ$-Cartier divisor on $X$ such that they generate $\WDiv_\mbR(X)$ modulo numerical equivalence over $U$. Let $H=H_1+H_2+\cdots+H_q$. By \cite[Lemma 3.7.4]{BCHM10} we may assume that if $\Delta\in\mcC$, then $\Delta$ contains the support of $H$. Let $W$ be the affine subspace of $\WDiv_\mbR(X)$ spanned by $V$ and the support of $H$. Let $\mcC'$ be a rational polytope in $\mcL_A(W)$ containing $\mcC$ in its interior such that $K_X+\Delta$ is KLT for all $\Delta\in\mcC'$.\\
Then by Proposition \ref{pro:polytope-finite} there are finitely many rational maps $\phi_i:X\rtmap Y_i$ $, 1\<i\<k$ over $U$, such that for any $\Delta'\in\mcC'\cap\mcE_{A, \pi}(W)$ there exists an index $1\<j\<k$ such that $\phi_j$ is a log minimal model of $K_X+\Delta'$ over $U$. By Corollary \ref{cor:weak-polytope} for each index $1\<i\<k$ there are finitely many projective contractions $f_{i, m}:Y_i\to Z_{i, m}$ over $U$ such that if $\Delta'\in\mcW_{\phi_i,A,\pi}(W)$ and there is a contraction $f:Y_i\to Z$ over $U$, with
\[K_{Y_i}+\Gamma_i=K_{Y_i}+\phi_{i,*}\Delta'\sim_{\mbR, U}f^*D,\] 	
	for some ample$/U$ $\mbR$-divisor $D$ on $Z$, then there is an index $(i, m)$ and an isomorphism $\xi:Z_{i, m}\to Z$ such that $f=\xi\circ f_{i, m}$. Let $\psi_j:X\rtmap Z_j$, $1\<j\<l$ be the finitely many rational maps obtained by composing every $\phi_i$ with every $f_{i, j}$.
	Pick $\Delta\in\mcC$ and let $\psi:X\rtmap Z$ be a weak log canonical model of $K_X+\Delta$ over $U$. Then $K_Z+\Theta$ is KLT and nef over $U$, where $\Theta=\psi_*\Delta$. Since $K_Z+\Theta$ is KLT, by \cite[Theorem 1.6]{Bir16} $Z$ has a $\mbQ$-factorization $\eta:Y'\to Z$, where $\eta$ is a small birational morphism and $Y'$ is $\mbQ$-factorial. Then by \cite[Lemma 3.6.12]{BCHM10} we may find $\Delta'\in\mcC'\cap\mcE_{A,\pi}(W)$ such that $\psi$ is an ample model of $K_X+\Delta'$ over $U$. Pick an index $1\<i\<k$ such that $\phi_i$ is a log minimal model of $K_X+\Delta'$ over $U$. By \cite[Lemma 3.6.6(4)]{BCHM10} there exists a contraction $f:Y_i\to Z$ such that
	\[K_{Y_i}+\Gamma_i=f^*(K_Z+\Theta'),\] 
	where $\Gamma_i=\phi_{i*}\Delta'$ and $\Theta'=\psi_*\Delta'$. As $K_Z+\Theta'$ is ample over $U$, it follows that there is an index $m$ and isomorphism $\xi:Z_{i, m}\to Z$ such that $f=\xi\circ f_{i, m}$. But then
	\[\psi=f\circ \phi_i=\xi\circ f_{i, m}\circ\phi_i=\xi\circ\psi_j,\] 
	for some index $1\<j\<l$.

	\end{proof}

\begin{corollary}\emph{(Finiteness of Weak Log Conical Models)}\label{cor:finite-weak-mmp}
	Let $\pi:X\to U$ be a projective contraction between two normal projective varieties with $\dim X=3$ and $\chr p>5$. Fix a general ample$/U$ $\mbQ$-divisor $A\>0$. Let $V$ be a finite dimensional affine subspace of $\WDiv_\mbR(X)$ which is defined over the rationals. Suppose that there is a KLT pair $(X, \Delta_0)$.\\
	Then there are finitely many birational maps $\psi_j:X\rtmap Z_j$ over $U$, $1\<j\<l$ such that if $\psi:X\rtmap Z$ is a weak log canonical model of $K_X+\Delta$ over $U$, for some $\Delta\in\mcL_A(V)$, then there is an index $1\<j\<l$ and an isomorphism $\xi:Z_j\to Z$ such that $\psi=\xi\circ\psi_j$. 
	\end{corollary}

\begin{proof}
	This result corresponds to Lemma 7.3 in \cite{BCHM10}.\\
	Since $\mcL_A(V)$ is a rational polytope, the statement follows from Theorem \ref{thm:polytope-weak-mmp}.

	\end{proof}

\begin{proof}[Proof of Theorem \ref{thm:finite-mmp}]
First we prove $(1)$ and $(2)$. Since ample models are unique by \cite[Lemma 3.6.6(1)]{BCHM10}, by Corollary \ref{cor:finite-weak-mmp} and \ref{cor:weak-polytope} it suffices to prove that if $\Delta\in\mcE_{A,\pi}(V)$, then $K_X+\Delta$ has both a log minimal model over $U$ and an ample model over $U$.\\
By \cite[Lemma 3.7.5 and 3.6.9]{BCHM10} we may assume that $K_X+\Delta$ is KLT. Then \cite[Theorem 1.2]{Bir16} gives the existence of a log minimal model of $K_X+\Delta$ over $U$, and the existence of the ample model follows from Lemma \ref{lem:semi-ample-to-ample}.\\

Part $(3)$ follows as in the proof of Corollary \ref{cor:weak-polytope}. Indeed if $\Delta_1, \Delta_2,\ldots, \Delta_k$ are the vertices of $\mcW_{\phi, A, \pi}(V)$ for a birational contraction $\phi:X\rtmap Y$, and $\Delta$ and $\Delta'$ are two divisors lying in the interior of $\mcW_{\phi, A, \pi}(V)$, then for a given (fixed) $0\<l\<k$ we can write $\Delta=\sum\mu_j\Delta_j$ and $\Delta'=\sum\mu'_j\Delta_j$ for some $\mu_l>0$. Therefore from the base-point free theorem \cite[Theorem 1.2]{BW17} it follows that a curve $C$ is contracted by $K_Y+\Gamma=K_Y+\phi_*\Delta$ if and only if it is contracted by $K_Y+\Gamma'=K_Y+\phi_*\Delta'$. In particular, the interior of $\mcW_{\phi, A, \pi}(V)$ is contained in a single ample model $\mcA_{\psi, A, \pi}(V)$ for some projective contraction $\psi:Y\to Z$. Therefore $\mcW_{\phi, A, \pi}(V)\subset\bar{\mcA}_{\phi, A, \pi}(V)$.\\ 

 Part $(4)$ follows combining Part $(1), (2)$ and $(3)$.

\end{proof}

\begin{proof}[Proof of Corollary \ref{cor:adjoint-rings}]

	Let $V$ be the finite dimensional affine subspace of $\WDiv(X)_\mbR$ generated by the irreducible components of $\Delta_1, \Delta_2,\ldots, \Delta_k$. Then by Theorem \ref{thm:finite-mmp} there exist finitely many rational maps $\phi_i:X\rtmap Y_i$ over $U$, $1\<i\<q$, such that for every $\Delta\in\mcE_{A,\pi}(V)$ there is an index $1\<j\<q$ such that $\phi_j$ is a log minimal model of $K_X+\Delta$ over $U$. Let $\mcC\subset\mcL_A(V)$ be the (rational) polytope spanned by $\Delta_1, \Delta_2,\ldots, \Delta_k$ and let
	\[
		\mcC_j=\mcW_{\phi_j, A, \pi}(V)\cap\mcC.
	\]
	Then $\mcC_j$ is a rational polytope. Note that the ring $\mcR(\pi, D^\bullet)$ is finitely generated if and only if the rings corresponding to the (rational) polytope $\mcC_j$ are finitely generated for all $j=1, 2,\ldots, q$. Therefore replacing $\Delta_1, \Delta_2,\ldots, \Delta_k$ by the vertices of $\mcC_j$ we may assume that $\mcC=\mcC_j$ and $\phi:X\rtmap Y$ is a log minimal model for all $\Delta_1, \Delta_2,\ldots, \Delta_k$. Let $\pi':Y\to U$ be the induced morphism. Let $\Gamma_i=\phi_*\Delta_i$ for all $1\<i\<k$. Let $g:W\to X$ and $h:W\to Y$ be a resolution of the graph of $\phi:X\rtmap Y$.
	\[
		\xymatrixcolsep{3pc}\xymatrixrowsep{3pc}\xymatrix{
		& W\ar[dl]_g\ar[dr]^h &\\
		X\ar@{-->}[rr]^\phi\ar[dr]_\pi && Y\ar[dl]^{\pi'}\\
		& U &
		}
	\]
	
Then we have
\begin{equation*}
	g^*(K_X+\Delta_i)=h^*(K_Y+\Gamma_i)+F_i
\end{equation*}	
for $1\<i\<k$.\\
Note that $F_i\>0$ is an effective $h$-exceptional divisor, since $\phi:X\rtmap Y$ is a log minimal model$/U$ for $\Delta_i$, $1\<i\<k$.
Let $m>0$ be positive integer such that $G_i=m(K_Y+\Gamma_i)$ and $D_i=m(K_X+\Delta_i)$ are both Cartier for all $1\<i\<k$. Then from the projection formula it follows that 
\[
\mcR(\pi, D^\bullet)\cong\mcR(\pi', G^\bullet).	
\]
Therefore replacing $X$ by $Y$ we may assume that $X$ is $\mbQ$-factorial and $K_X+\Delta_i$ is KLT and nef over $U$ for all $1\<i\<k$. Since $\Delta_i$ is big for all $1\<i\<k$, by \cite[Theorem 1.2]{BW17} $K_X+\Delta_i$ is semi-ample for all $1\<i\<k$. Therefore it follows that the ring $\mcR(\pi, D^\bullet)$ is finitely generated.

\end{proof}

\section{The duality of pseudo-effective divisors and movable curves}
In this section we will work on projective varieties of arbitrary dimension and over an algebraically closed ground field $k=\overline{k}$ of arbitrary characteristic. We will prove Theorem \ref{thm:cone-duality} and \ref{thm:covering-rational-curves} here.

\begin{definition}(Movable curves, strongly movable curves and nef curves)
	Let $X$ be a projective variety. An irreducible curve $C$ is called \emph{movable} if there exists an algebraic family of irreducible curves $\{C_t\}_{t\in T}$ such that $C=C_{t_0}$ for some $t_0\in T$ and $\cup_{t\in T} C_t\subset X$ is dense in $X$.\\
	A class $\gamma\in N_1(X)_\mbR$ is called movable if there exists a movable curve $C$ such that $\gamma=[C]$ in $N_1(X)_\mbR$.\\ 
	
	An irreducible curve $C$ is called \emph{strongly movable} if there exists a projective birational morphism $f:X'\to X$ and ample divisors $H_1', H_2',\ldots, H_{n-1}'$ on $X'$ such that $C=f_*(H_1'\cap H_2'\cap\ldots \cap H_{n-1}')$, where $n=\dim X=\dim X'$.\\
	A class $\gamma\in N_1(X)_\mbR$ is called strongly movable if there exists a strongly movable curve $C$ such that $\gamma=[C]$ in $N_1(X)_\mbR$.\\ 
	
	An irreducible curve $C$ is called a \emph{nef curve} if $D\cdot C\>0$ for every effective Cartier divisor $D\>0$. A class $\gamma\in N_1(X)_\mbR$ is called nef is there exists a nef curve $C$ such that $\gamma=[C]$.\\  
	\end{definition}

\begin{definition}\mbox{(Cone of movable, strongly movable and nef curves)}
	Let $X$ be a projective variety. The closure in $N_1(X)_\mbR$ of the cone of effective classes of movable curves 
	\[\NM(X)=\overline{\left\{\sum a_i\gamma_i: a_i\>0 \mbox{ and } \gamma_i\in N_1(X)_\mbR \mbox{ is movable }\right\}} \]
	is called the \emph{cone of movable curves}.\\
	
	The closure in $N_1(X)_\mbR$ of the cone of effective classes of strongly movable curves 
		\[\SNM(X)=\overline{\left\{\sum a_i\gamma_i: a_i\>0 \mbox{ and } \gamma_i\in N_1(X)_\mbR \mbox{ is strongly movable }\right\}} \]
		is called the \emph{cone of strongly movable curves}.\\ 
		
	The closure in $N_1(X)_\mbR$ of the cone of effective classes of nef curves 
		\[\NF(X)=\overline{\left\{\sum a_i\gamma_i: a_i\>0 \mbox{ and } \gamma_i\in N_1(X)_\mbR \mbox{ is nef }\right\}} \]
		is called the \emph{cone of nef curves}.\\
	\end{definition}


The following theorem of Takagi on the existence of Fujita approximation in arbitrary characterisitc is one of the main ingredient of our proof.
 
\begin{theorem}\emph{(Fujita's approximation theorem)\cite[Corollary 2.16]{Tak07}}\label{thm:Takagi's-Theorem}
	Let $X$ be a projective variety defined over an algebraically closed field $k$ of arbitrary characteristic. Let $\xi\in N^1(X)_\mbR$ be a big divisor class. Then for any real number $\ve>0$ there exists a birational morphism $\mu:X'\to X$ from a projective variety $X'$ and a decomposition
	\[\mu^*(\xi)=a+e \]
	in $N^1(X')_\mbR$ such that 
	\begin{enumerate}
	\item $a$ is an ample class and $e$ is effective, and
	\item $\vol_{X'}(a)>\vol_{X}(\xi)-\ve$.\\
	\end{enumerate}
	\end{theorem}
 
\begin{theorem}\label{thm:Fujita-orthogonality}
	Let $X$ be a projective variety of dimension $n$ over an algebraically closed field $k$ of arbitrary characteristic. Let $\xi\in N^1(X)_\mbR$ be a big divisor class. Consider a Fujita approximation of $\xi$:
	\[\mu: X'\to X, \qquad \mu^*\xi=a+e. \]
Let $h\in N^1(X)_\mbR$ be an ample class such that $h\pm\xi$ are both ample. Then
 \[ (a^{n-1}\cdot e)^2_{X'}\<20\cdot(h^n)_X\cdot \biggl(\vol_{X}(\xi)-\vol_{X'}(a)\biggr).\]\\
	\end{theorem}

\begin{proof}
This is Theorem 11.4.21 in \cite{Laz04a}. The main ingredients of the proof of \cite[Theorem 11.4.21]{Laz04b} are the Fujita's approximation theorem, Hodge type inequalities \cite[Corollary 1.6.3, Lemma 11.4.22]{Laz04a}, and the continuity of volume \cite[Theorem 2.2.44, Example 2.2.47]{Laz04a}. The Fujita's approximation theorem is known in positive characteristic due to \cite{Tak07} and the other two results are also known to hold in positive characteristic (their proofs in \cite{Laz04a} work in arbitrary characteristic). As a result, the proof of \cite[Theorem 11.4.21]{Laz04b} holds in arbitrary characteristic.\\
	\end{proof}

\begin{proof}[Proof of the Theorem \ref{thm:cone-duality}]
It is well known that $\Eff(X)\subset\SNM(X)^*$. By contradiction, assume that the inclusion $\Eff(X)\subsetneq\SNM(X)^*$ is strict. Then there exists a class $\xi\in N^1(X)_{\mbR}$ such that 
\[\xi\in\mbox{ boundary}\left(\Eff(X)\right) \mbox{  and  } \xi\in\mbox{ interior}\left(\SNM(X)^*\right). \]	
Fix an ample class $h$ such that $h\pm 2\xi$ is ample. Since $\xi$ lies in the interior of $\SNM(X)^*$, there exists $\ve>0$ such that $\xi-\ve h\in\SNM(X)^*$.	In particular,
\begin{equation}\label{eqn:main-inequality}
	\frac{(\xi\cdot\gamma)}{(h\cdot\gamma)}\>\ve
	\end{equation}
for every strongly movable class $\gamma$ on $X$. Now consider the class 
\[\xi_\delta=\xi+\delta h \quad\mbox{ for }\quad \delta>0. \]	
Since $\xi$ is pseudo-effective, $\xi_\delta$ is big for all $\delta>0$. For small $\delta>0$ consider a Fujita approximation (by Theorem \ref{thm:Takagi's-Theorem})
\[\mu_\delta:X'_{\delta}\to X, \qquad \mu^*(\xi_\delta)=a_\delta+e_\delta\] 	
such that
\begin{equation}\label{eqn:Fujita-approx}
	\vol_{X'_\delta}(a_\delta)\>\vol_{X}(\xi_\delta)-\delta^{2n}.
	\end{equation}	
Since $\delta^{2n}$ is a polynomial of degree $2n$ in $\delta$ and $\vol_{X}(\xi_\delta)=(\xi+\delta h)^n$ is a polynomial of degree $n$ in $\delta$ and $(h^n)>0$, for $\delta>0$ sufficiently small we may assume that $0<\delta^{2n}<\frac{1}{2}\vol_{X}(\xi_\delta)$. In particular,
\begin{equation}\label{eqn:ample-part-Fujita-approx}
	\vol_{X'_\delta}(a_\delta)\>\frac{1}{2}\cdot\vol_{X}(\xi_\delta)\>\frac{\delta^n}{2}\cdot (h^n).\\ 
	\end{equation} 	
Consider the strongly movable class
\[\gamma_\delta=\mu_{\delta, *}\left(a_{\delta}^{n-1}\right).\]	
Then by the projection formula and \cite[Corollary 1.6.3(ii)]{Laz04a}, we get
\begin{equation}\label{eqn:Hodge-with-ample}
	\begin{split}
		(h\cdot \gamma_\delta)_X &=(\mu^*_\delta(h)\cdot a^{n-1}_\delta)_{X'_\delta}\\
								 &\> (h^n)^{1/n}_X\cdot (a_\delta^n)^{(n-1)/n}_{X'_\delta}.
		\end{split}
		\end{equation}
On the other hand,
			
	\begin{equation}\label{eqn:Hodge-with-pseudo-effective}
		\begin{split}
			(\xi\cdot\gamma_\delta)_X &\<(\xi_\delta\cdot\gamma_\delta)_X\\
									  &=(\mu^*_\delta(\xi_\delta)\cdot a_\delta^{n-1})_{X'_{\delta}}\\
									  &=(a_{\delta}^n)_{X'_\delta}+(e_\delta\cdot a_\delta^{n-1})_{X'_\delta}.
			\end{split}	
\end{equation}	

Now $h\pm\xi_\delta$ is ample provided that $\delta<\frac{1}{2}$. Therefore by Theorem \ref{thm:Fujita-orthogonality} and \eqref{eqn:Fujita-approx} we get,
\begin{equation}\label{eqn:Fujita-orthogonality}
	\begin{split}	
	(e_\delta\cdot a^{n-1}_\delta)_{X'_\delta} &\<\left(20\cdot (h^n)_X\cdot (\vol_X(\xi_\delta)-\vol_{X'_\delta}(a_\delta))\right)^{1/2}\\
											   &\< C_1\cdot \delta^n,
	\end{split}
	\end{equation}
where $C_1=20\cdot(h^n)_X>0$ is independent of $\delta$.\\

Now from \eqref{eqn:ample-part-Fujita-approx}, \eqref{eqn:Hodge-with-ample}, \eqref{eqn:Hodge-with-pseudo-effective} and \eqref{eqn:Fujita-orthogonality} we get
\begin{equation}\label{eqn:main-inequality-contradiction} 
	\begin{split}		
	0\<\frac{(\xi\cdot\gamma_\delta)}{(h\cdot\gamma_\delta)} &\<\frac{(a^n_\delta)_{X'_\delta}+C_1\cdot \delta^n}{(h^n)^{1/n}_X\cdot (a^n_\delta)^{(n-1)/n}_{X'_\delta}}\\
	&\<\frac{1}{(h^n)^{1/n}_X}\cdot (a^n_\delta)^{1/n}_{X'_\delta}+\frac{C_1\cdot\delta^n}{(h^n)^{1/n}_X}\cdot\frac{2^{(n-1)/n}}{\delta^{n-1}\cdot (h^n)^{(n-1)/n}_X}\\
	&\< C_2\cdot (a^n_\delta)^{1/n}_{X'_\delta}+C_3\cdot\delta,
	\end{split}
	\end{equation}
where $C_2$ and $C_3$ are constants independent of $\delta$.\\

Now recall that $\xi$ lies on the boundary of the big cone. Therefore by the continuity of volume \cite[Theorem 2.2.44]{Laz04a}
\[\lim_{\delta\to 0}\vol_X(\xi_\delta)=\vol_X(\xi)=0,\]	
and hence also
\[\lim_{\delta\to 0}\vol_{X'_\delta}(a_\delta)=\lim_{\delta\to 0}(a^n_{\delta})_{X'_\delta}=0.\]
Thus from \eqref{eqn:main-inequality-contradiction} we see that $\frac{(\xi\cdot\gamma_\delta)}{(h\cdot\gamma_\delta)}\to 0$ as $\delta\to 0^+$. But this contradicts \eqref{eqn:main-inequality}, and completes the proof.\\

	\end{proof}

\begin{corollary}\label{cor:cone-equality}
Let $X$ be a projective variety defined over an algebraically closed field of arbitrary characteristic. Then the cone of movable curves, the strongly movable curves and the cone of nef curves all coincide, i.e., $\NM(X)=\SNM(X)=\NF(X)$.
\end{corollary}

\begin{proof}
From the definition of nef curves it is clear that $\NF(X)=\Eff(X)^*$, and $\Eff(X)^*=\SNM(X)$ by Theorem \ref{thm:cone-duality}. Thus we only need to prove that $\NM(X)=\SNM(X)$. The inclusion $\SNM(X)\subset\NM(X)$ is clear. We will prove the other inclusion. Let $\gamma\in\NM(X)$. Then there exists a movable curve $C$ such that $\gamma=[C]$. By Theorem \ref{thm:cone-duality} it is enough to show that $D\cdot\gamma\>0$ for all effective Cartier divisors. Since $C$ belongs to an algebraic family of curves $\{C_t\}_{t\in T}$ such that $\cup_{t\in T}C_t$ covers a dense subset of $X$, we can find a curve $C_{t_1}$ in this family such that $C_{t_1}\nsubseteq\Supp (D)$. Thus $D\cdot\gamma=D\cdot C_{t_1}\>0$, i.e. $\gamma\in{\Eff(X)}^*=\SNM(X)$.\\	
	
	\end{proof}

\begin{proof}[Proof of Theorem \ref{thm:covering-rational-curves}]
If there exists an algebraic family of $K_X$-negative rational curves covering a dense subset of $X$, then from Theorem \ref{thm:cone-duality} and Corollary \ref{cor:cone-equality} it follows that $K_X$ it not pseudo-effective.\\

Now assume that $K_X$ is not pseudo-effective. Then by Theorem \ref{thm:cone-duality} and Corollary \ref{cor:cone-equality}, there exist a movable class $\gamma\in\NM(X)$ and an algebraic family of irreducible curves $\{C_t\}_{t\in T}$ representing $\gamma$ such that $K_X\cdot C_t<0$ for all $t\in T$ and $\cup_{t\in T}C_t\subset X$ is dense in $X$. We fix a very ample divisor $H$ in $X$. Then by \cite[Theorem 1.1]{Miy95} there exist rational curves $\{C'_s\}_{s\in S}$ of bounded degree through every point of $\cup_{t\in T}C_t\subset X$ such that 
\begin{equation}
	H\cdot C'_s\<\frac{2\dim X\cdot(H\cdot\gamma)}{-K_X\cdot\gamma}\quad\mbox{and}\quad 0<-(K_X\cdot C'_s)\<\dim X+1 \mbox{ for all } s\in S.
	\end{equation}
By \cite[Corollary 1.19(3)]{KM98} there are finitely many subclasses of these rational curves, say $\{C'_s\}_{s\in S_i}$, $1\<i\<n, S_i\subset S \mbox{ and }\coprod_{i=1}^n S_i=S$, such that for each fixed $i$, any two curves in $\{C'_s\}_{s\in S_i}$ are numerically equivalent. Now since $\cup_{s\in S}C'_s\subset X$ is dense in $X$, it follows that one of these subclasses, say $\{C'_\lambda\}_{\lambda\in\Lambda}$, $\Lambda=S_{i_k}$ for some $i_k\in\{1, 2,\ldots, n\}$ has the property that $\cup_{\lambda\in\Lambda}C'_\lambda\subset X$ is dense in $X$. Let $d:=H\cdot C'_\lambda$. Then the curves $\{C'_\lambda\}_{\lambda\in\Lambda}$  belong to the family $\Univ_d(\mbP^1, X)\to \Hom_d(\mbP^1, X)$, where $\Hom_d(\mbP^1, X)$ is the scheme of degree $d$ morphisms $\mbP^1\to X$. Let $V'\subset\Hom_d(\mbP^1, X)$ be the connected component of $\Hom_d(\mbP^1, X)$ which contains all the points corresponding to the curves $\{C'_\lambda\}_{\lambda\in\Lambda}$ (since all $C'_\lambda$'s are numerically equivalent, they are contained in a connected component). Let $\mathcal{U}'=V'\times_{\Hom_d(\mbP^1, X)}\Univ_d(\mbP^1, X)$. Then $\mathcal{U}'\to V'$ is a (flat) family of rational curves $\Gamma_t$ such that $K_X\cdot \Gamma_t<0$ and $\cup_{t\in V'}\Gamma_t\subset X$ is dense in $X$. Finally, let $V$ be an irreducible component of $V'$ such that $\cup_{t\in V}\Gamma_t\subset X$ is dense in $X$. Set $\mathcal{U}=V\times_{V'}\mathcal{U}'=V\times_{\Hom_d(\mbP^1, X)}\Univ_d(\mbP^1, X)$. Then $\mathcal{U}\to V$ is a (flat) algebraic family of rational curves satisfying the required conditions.\\

	\end{proof}

\begin{proof}[Proof of Corollary \ref{cor:uniruled-criterion}]
Following the notations as in the proof of Theorem \ref{thm:covering-rational-curves} we see that the evaluation map $\ev:\mbP^1\times_k V\to X$ is a dominant morphism, where $V$ is an irreducible variety. Thus by \cite[Remark 4.2(2)]{Deb01} $X$ is uniruled. 	
\end{proof}

\section{The structure of the nef cone of curves}
In this section we prove the structure theorem for nef cone of curves. It gives a partial answer to Batyrev's Conjecture \ref{con:batyrev} in positive characteristic.\\

We define coextremal rays and bounding divisors as in \cite{Leh12}.
\begin{definition}
	Let $\alpha$ be a class in $\NM(X)$. A \emph{coextremal} ray  $\mbR_{\>0}\alpha\subset N_1(X)$ is a $(K_X+\Delta)$-negative ray of $\NM(X)$ which is extremal for $\NE(X)_{K_X+\Delta\>0}+\NM(X)$; equivalently it satisfies the following properties:
	\begin{enumerate}
		\item $(K_X+\Delta)\cdot\alpha<0$.
		\item If $\beta_1, \beta_2\in\NE(X)_{K_X+\Delta\>0}+\NM(X)$ and $\beta_1+\beta_2\in\mbR_{\>0}\alpha$, then $\beta_1, \beta_2\in\mbR_{\>0}\alpha$.\\
	\end{enumerate}
\end{definition}

\begin{definition}
	A non-zero $\mbR$-Cartier divisor $D$ is called a bounding divisor if it satisfies the following properties:
	\begin{enumerate}
		\item $D\cdot\alpha\>0$ for every class $\alpha$ in $\NE(X)_{K_X+\Delta\>0}+\NM(X)$.
		\item $D^\bot$ contains some coextremal ray.
	\end{enumerate}
For a subset $V\subset N_1(X)$, a bounding divisor $D$ is called a $V$-bounding divisor if $D\cdot\alpha\>0$ for all $\alpha\in V$.\\	
\end{definition}

We will need the following results first.

\begin{lemma}\label{lem:ltm-to-mfs}
 Let $(X, \Delta\>0)$ be a $\mbQ$-factorial projective KLT pair of dimension $3$ and $\chr p>5$. Suppose that $\Delta$ is a big $\mbR$-divisor. If $K_X+\Delta$ lies on the boundary of\ $\Eff(X)$, then there exists a birational contraction $\phi:X\rtmap X'$, a projective morphism $f:X'\to Y$ and an ample $\mbR$-divisor $L$ on $Y$ such that $K_{X'}+\Delta'\sim_{\mbR}f^*L$ and $-(K_{X'}+G')$ is $f$-ample for some KLT pair $(X', G')$, where $K_{X'}+\Delta'=\phi_*(K_X+\Delta)$.
\end{lemma}

\begin{proof}
	  Let $H\>0$ be an ample $\mbR$-divisor on $X$ such that $K_X+\Delta+H$ is KLT and nef (using Theorem \ref{thm:general-bertini}). Since $K_X+\Delta$ is pseudo-effective, by \cite[Theorem 1.6]{BW17} $(K_X+\Delta)$-MMP with the scaling of $H$ terminates with a log minimal model $\phi':X\rtmap X'$ such that $K_{X'}+\Delta'=\phi'_*(K_X+\Delta)$ is nef. Since $\Delta'$ is a big divisor, by perturbing $\Delta'$, from the base-point free theorem \cite[Theorem 1.2]{BW17} it follows that $K_{X'}+\Delta'$ is semi-ample. Therefore there exists a projective morphism $f:X'\to Y$ and an ample $\mbR$-divisor $L$ on $Y$ such that $K_{X'}+\Delta'\sim_\mbR f^*L$ (see \cite[Lemma 4.13]{Fuj11}). Write $\Delta'=\phi_*\Delta\num A'+B'$, where $A'$ is an ample $\mbR$-divisor. Then $(X', \Delta'+\ve B')$ is KLT for $0<\ve\ll 1$. In particular, $(X', (1-\ve)\Delta'+\ve B')$ is KLT. Let $G'=(1-\ve )\Delta'+\ve B'\>0$. Then $(X', G')$ is KLT and $K_{X'}+G'\num_f-\ve A'$.

	\end{proof}

\begin{remark}\label{rmk:ample-model}
	Note that the variety $Y$ in the lemma above is uniquely determined by $K_X+\Delta$, since it is the ample model of $(X, \Delta)$.\\ 
	\end{remark}

\begin{proposition}\label{pro:trimed-finiteness}
	 Let $(X, \Delta\>0)$ be a $\mbQ$-factorial projective KLT pair with $\dim X=3$ and $\chr p>5$. Suppose that $A\>0$ is a general ample $\mbQ$-divisor on $X$. Let $V$ be a finite dimensional subspace of $\WDiv_\mbR(X)$. Define
	\[\mcE_A=\{\Gamma\in V : \Gamma\>0 \mbox{ and } K_X+\Delta+A+\Gamma \mbox{ is KLT and pseudo-effective}       \}. \]
Then there are finitely many birational contractions $\phi_i:X\rtmap X_i$ for $1\<i\<k$ such that for every $\Gamma\in\mcE_A$ satisfying the property that  $K_X+\Delta+A+\Gamma$ lies on the boundary of $\Eff(X)$, there exists an index $1\<j\<k$ such that $\phi_j$ is a log minimal model of $K_X+\Delta+A+\Gamma$. Furthermore, corresponding to each $\phi_i$, there exists a unique log Fano fibration $g_i:X_i\to Y_i$ such that $K_{X_i}+{\phi_i}_*(\Delta+A+\Gamma)\sim_{\mbR}g^*_iL_i$, where $L_i$ is an ample $\mbR$-divisor on $Y_i$.\\  	
	\end{proposition}
	
\begin{proof}
	This result corresponds to Theorem 3.2 in \cite{Leh12}.\\
	For a given $\Gamma\in\mcE_A$ the existence of $\phi_i$ and $g_i$ is clear from Lemma \ref{lem:ltm-to-mfs}. Then from Theorem \ref{thm:finite-mmp} and Remark \ref{rmk:ample-model} it follows that there are only finitely many such $\phi_i$ and $g_i$ for all $\Gamma\in\mcE_A$.\\
	\end{proof}

\begin{proposition}\label{pro:nbhd-lemma}
	Let $(X, \Delta)$ be a $\mbQ$-factorial projective KLT pair with $\dim X=3$ and $\chr p>5$. Suppose that $B\>0$ is a big $\mbR$-divisor such that $(X, \Delta+B)$ is KLT. Then there exist an open neighborhood $U\subset N^1(X)$ of $[K_X+\Delta+B]\in N^1(X)$ and a finite set of movable curves $\{C_i\}_{i=1}^N$ on $X$ such that for every class $\alpha\in U$ which lies on the boundary of $\Eff(X)$ we have $\alpha\cdot C_i=0$ for some $i\in\{1, 2,\ldots, N\}$.
	\end{proposition}

\begin{proof}
	This result corresponds to Proposition 3.3 in \cite{Leh12}.\\
	Since $B$ is a big $\mbR$-divisor, $B\num H+E$ for some ample $\mbR$-divisor $H\>0$ and an effective $\mbR$-divisor $E$. Then $K_X+\Delta+B+\ve E$ is KLT for $0<\ve\ll 1$. In particular, $K_X+\Delta+(1-\ve)B+\ve E$ is KLT. Let $A\>0$ be an ample $\mbQ$-divisor such that $\ve H-A$ is ample and $K_X+\Delta+A+(1-\ve)B+\ve E$ is KLT, by Theorem \ref{thm:general-bertini}. Let $\{H_j\>0\}_{j=1}^m$ be a finite set of ample $\mbQ$-divisors such that the convex hull of the classes $[H_j]$'s contains an open set around $[\ve H-A]$ in $N^1(X)$. Let $U'$ be an open neighborhood of $[B-A]$ contained in the convex hull of $[B-\ve H+H_j]$'s. We apply \cite[Lemma 3.1]{Leh12} to $K_X+\Delta+A+(1-\ve)B+\ve E$ and $H_j$'s to obtain finitely many ample $\mbQ$-divisors $0\<W_j\sim_\mbQ H_j$. Let $V$ be the vector space of $\mbR$-divisors spanned by the irreducible components of $(1-\ve)B+\ve E$ and of the $W_j$'s. Therefore $V$ is a finite dimensional subspace of $\WDiv(X)_\mbR$ such that every class in $U'$ has an effective representative $\Gamma\in V$ with $(X, \Delta+A+\Gamma)$ KLT.\\
	
	Let $\Gamma\in V$ be a representative of a class in $U'$ such that $D=K_X+\Delta+A+\Gamma$ is KLT. We can run the $D$-MMP with the scaling of an ample divisor. If $D$ lies on the boundary of $\Eff(X)$, then by Proposition \ref{pro:trimed-finiteness} there exists a birational contraction $\phi_i:X\rtmap X_i$ and a log Fano fibration $g_i:X_i\to Z_i$ such that $K_{X_i}+{\phi_i}_*(\Delta+A+\Gamma)\sim_{\mbR} f^*L_i$ for some ample $\mbR$-divisor $L_i$ on $Z_i$. Let $C'_i$ be a  curve on a general fiber of $g_i$. Note that the exceptional locus $\Ex(\phi^{-1}_i)$ where $\phi^{-1}_i:X_i\rtmap X$ is not an isomorphism intersects the general fiber of $g_i$ along at least codimension $2$ subsets (see \cite[Lemma 2.4]{BW17}). Therefore by choosing $C'_i$ sufficiently general we see that $\phi_i:X_i\rtmap X$ is an isomorphism in a neighborhood of $C'_i$ and $C'_i$ belongs to a family of curves dominating $X_i$. Let $C_i$ be the image of $C'_i$ under $\phi^{-1}_i$. Then $C_i$ is a movable curve on $X$ and $D\cdot C_i=(K_{X_i}+{\phi_i}_*(\Delta+A+\Gamma))\cdot C'_i=f^*L_i\cdot C'_i=0$.\\
	
	Now by Proposition \ref{pro:trimed-finiteness} there are finitely many log Fano bibrations $g_i:X_i\to Z_i$ for all $\Gamma\in V$. Therefore there are finitely many movable curves $\{C_i\}$ on $X$ satisfying the properties as in the previous paragraph. Set $U=U'+[K_X+\Delta+A]$. Then $U$ is an open neighborhood of $[K_X+\Delta+A]$ and the curves $C_i$'s have the required properties.

	\end{proof}

\begin{corollary}\label{cor:global-nbhd-lemma}
Let $(X, \Delta)$ be a $\mbQ$-factorial projective KLT pair with $\dim X=3$ and $\chr p>5$. Suppose that $\mcS\subset\Eff(X)$ is a set of divisor classes satisfying the following properties:
\begin{enumerate}
	\item $\mcS$ is closed.
	\item For each element $\beta\in\mcS$, there is some big effective divisor $B$ such that $(X, \Delta+B)$ is KLT and $[K_X+\Delta+B]=c\beta$ for some $c>0$.
\end{enumerate}

Then there are finitely many movable curves $\{C_i\}$ such that every class $\alpha$ which lies on the boundary of $\Eff(X)$ satisfies $\alpha\cdot C_i=0$ for some $i$.
\end{corollary}

\begin{proof}
	With the help of Proposition \ref{pro:nbhd-lemma}, the same proof as in \cite[Corollary 3.5]{Leh12} works.\\ 
\end{proof}

\begin{proposition}\label{pro:v-bounding-divisor}
Let $(X, \Delta)$ be a $\mbQ$-factorial projective KLT pair with $\dim X=3$ and $\chr p>5$. Let $V$ be a closed convex cone containing $\NE(X)_{K_X+\Delta=0}-\{0\}$ in its interior. Then there is a finite set of movable curves $\{C_i\}$ such that for any $V$-bounding divisor $D$ there is some $C_i$ for which $D\cdot C_i=0$.	
\end{proposition}

\begin{proof}
	This result corresponds to Proposition 4.4 in \cite{Leh12}.\\
Let $\mcS$ be the set of all $V$-bounding divisors. We may assume that $\mcS$ is non-empty, otherwise there is nothing to prove. Note that a non-zero $\mbR$-Cartier divisor $D$ is $V$-bounding if and only if $D$ lies on the boundary of $\Eff(X)$ and satisfies the closed condition 
\[
D\cdot\alpha\>0 \mbox{ for every class } \alpha\in\NE(X)_{K_X+\Delta\>0}+V+\NM(X).	
\]  
Now by \cite[Lemma 4.3]{Leh12} there exists an ample $\mbR$-Cartier divisor $A_D\>0$ and a positive real number $\delta_D>0$ such that 
\[
	\frac{1}{\delta_D}D=(K_X+\Delta)+\frac{1}{\delta_D}A_D.
\]
Let $A'\>0$ be an ample $\mbR$-divisor such that $A'\sim_\mbR\frac{1}{\delta_D} A_D$ and $(X, \Delta+A')$ is KLT (see Theorem \ref{thm:general-bertini}). Then we have 
\[
	[K_X+\Delta+A']=\frac{1}{\delta_D}[D].
\]
Therefore by Corollary \ref{cor:global-nbhd-lemma} there exists finitely many movable curves $\{C_i\}$ satisfying the required conditions.
	
\end{proof}

\begin{theorem}\label{thm:Q-dlt-case}
	Let $(X, \Delta)$ be a $\mbQ$-factorial projective DLT pair with $\dim X=3$ and $\chr p>5$. Let $V$ be a closed convex cone containing $\NE(X)_{K_X+\Delta=0}-\{0\}$ in its interior. Then there exist finitely many movable curves $C_i$ such that
	\begin{equation}\label{eqn:Q-cone-comparison}
		\NE(X)_{K_X+\Delta\>0}+V+\NM(X)=\NE(X)_{K_X+\Delta\>0}+V+\sum_{i=1}^N\mbR_{\>0}[C_i].
	\end{equation}
\end{theorem}

\begin{proof}
	This result is a variant of Corollary 4.5 in \cite{Leh12}.\\
Without loss of generality we may assume by shrinking $V$ if necessary that $\NE(X)+V$ does not contain any $1$-dimensional subspace of $N_1(X)$. Then there exists an ample divisor $A$ which is positive on $V-\{0\}$.\\

First we reduce the problem to the KLT case. Let $H\>0$ be an ample $\mbQ$-divisor on $X$. Since $X$ is $\mbQ$-factorial, $H+\ve\Delta$ is ample for $0<\ve\ll 1$. Let $0\<A\sim_\mbR H+\ve\Delta$ be an ample divisor which avoids all DLT centers of $(X, \Delta)$ as well as all irreducible components of $\Delta$. Then $(X, \Delta+\delta A)$ is DLT for $1<\delta\ll 1$. In particular, $(X, (1-\delta\ve)\Delta+\delta A)$ is KLT and $K_X+(1-\delta\ve)\Delta+\delta A\sim_\mbR K_X+\Delta+\delta H$. Choose $\delta>0$ sufficiently small such that $\NE(X)_{K_X+(1-\delta\ve)\Delta+\delta A=0}$ is contained in the interior of the cone $V$. We will show that if we assume the statement for $((1-\delta\ve)\Delta+\delta A, V)$, then it holds for $(\Delta, V)$.\\
Let $\alpha\in\NM(X)$. If $\alpha\in \NE(X)_{K_X+(1-\delta\ve)\Delta+\delta A\>0}+V+\sum\mbR_{\>0}[C_i]$, then $\alpha=\beta+\gamma+\sum a_i[C_i]$, where $\beta\in \NE(X)_{K_X+(1-\delta\ve)\Delta+\delta A\>0}, \gamma\in V$ and $a_i\>0$ for all $i$. Note that $(K_X+(1-\delta\ve)\Delta+\delta A)\cdot\beta\>0$ implies that $(K_X+\Delta)\cdot\beta\>0$ by letting $\delta\to 0^+$. Therefore by replacing $(X, \Delta)$ by $(X, (1-\delta\ve)\Delta+\delta A)$ we may assume that $(X, \Delta\>0)$ is KLT.\\

Let $\{C_i\}_{i=1}^N$ be the finite set of curves obtained in Proposition \ref{pro:v-bounding-divisor}. It is obvious that the right hand side of \eqref{eqn:Q-cone-comparison} is contained in the left hand side. We will show the other inclusion. Let $\alpha\in\NM(X)$ such that it is not contained in $\NE(X)_{K_X+\Delta\>0}+V+\sum\mbR_{\>0}[C_i]$. Then there exists a $\mbR$-divisor $B$ such that $B\cdot \gamma\>0$ for all $\gamma\in \NE(X)_{K_X+\Delta\>0}+V+\sum\mbR_{\>0}[C_i]$ but $B\cdot \alpha<0$.\\

 Let $A$ be an ample divisor which is positive on $V-\{0\}$. Set $\ve=\max\{t>0:A+tB \mbox{ is pseudo-effective }\}$. Then $\ve>0$. Furthermore, from the discussion in the proof of the Proposition \ref{pro:v-bounding-divisor} it follows that $A+\ve B$ is a $V$-bounding divisor. However, $(A+\ve B)\cdot C_i>0$ for all $i\in\{1, 2,\ldots, N\}$, a contradiction to the Proposition \ref{pro:v-bounding-divisor}.

\end{proof}

\begin{theorem}\label{thm:dlt-case}
	Let $(X, \Delta)$ be a projective DLT pair with $\dim X=3$ and $\chr p>5$. Let $V$ be a closed convex cone containing $\NE(X)_{K_X+\Delta=0}-\{0\}$ in its interior. Then there are finitely many movable curves $C_i$ such that 
	\[
		\NE(X)_{K_X+\Delta\>0}+V+\NM(X)=\NE(X)_{K_X+\Delta\>0}+V+\sum_{i=1}^N\mbR_{\>0}[C_i].
	\]
\end{theorem}

\begin{proof}
	This result corresponds to Theorem 4.7 in \cite{Leh12}.\\
The statement is vacuously true if $K_X+\Delta$ is pseudo-effective. So we may assume that $K_X+\Delta$ is not pseudo-effective. We will complete the proof in two steps.\\

First we reduce the problem to a $\mbQ$-factorial DLT pair. By Lemma \ref{lem:dlt-to-Q-factorization} there exists a small birational morphism $\pi:Y\to X$ and a $\mbQ$-factorial DLT pair $(Y, \Gamma\>0)$ such that
\[
	K_Y+\Gamma=\pi^*(K_X+\Delta).
\]
Since the map $\pi_*:N_1(Y)\to N_1(X)$ is linear, $\pi^{-1}_*V$ is closed and convex. Furthermore, $\NE(Y)_{K_Y+\Gamma=0}$ is contained in the interior of $\pi^{-1}_*V$, since $\pi_*$ on curves is dual to $\pi^*$ on divisors. Then by Theorem \ref{thm:Q-dlt-case} we have
\[
	\NE(Y)_{K_Y+\Gamma\>0}+\pi^{-1}_*V+\NM(Y)=\NE(Y)_{K_Y+\Gamma\>0}+\pi^{-1}_*V+\sum_{i=1}^N\mbR_{\>0}[C'_i],
\]
where $\{C'_i\}$ are movable curves on $Y$.\\

By \cite[Lemma 2.1]{Leh12} pushing forward the above relation by $\pi_*$ we get
\[
	\NE(X)_{K_X+\Delta\>0}+V+\NM(X)=\NE(X)_{K_X+\Delta\>0}+V+\sum_{i=1}^N\mbR_{\>0}[C_i],
\]
where $C_i=\pi_*C'_i$.\\
Since birational push-forward of a movable curve is again movable, this completes the proof.

\end{proof}

\begin{proof}[Proof of Theorem \ref{thm:nef-cone-structure}]
	Let $\{V_j\}$ be a countable collection of nested closed convex cones containing $\NE(X)_{K_X+\Delta=0}-\{0\}$ in their interiors such that
	\[
		\bigcap_j V_j=\NE(X)_{K_X+\Delta=0}.
	\]
Let $\mcC_j$ be the finite set of movable curves corresponding to $V_j$ obtained in Theorem \ref{thm:dlt-case}. Note that all the curves in $\mcC'_j$ lie on the boundary of $\NM(X)$, but not all of them generate extremal rays of $\NM(X)$. By removing those redundant curves we may assume that each curve in $\mcC_j$ generates a co-extremal ray. Define $\mcC=\cup_j\mcC_j$. Then $\mcC$ is at most countable.\\

By contradiction assume that $\alpha\in\NM(X)$ but
\[
	\alpha\not\in\NE(X)_{K_X+\Delta\>0}+\overline{\sum_\mcC\mbR_{\>0}[C_i].}
\]	
Since this cone is closed and convex, there is a convex open neighborhood $U$ of the cone which does not contain $\alpha$. Then from our construction it follows that $V_j\subset U$ for $j$ sufficiently large. In particular,
\[
	\alpha\not\in\NE(X)_{K_X+\Delta\>0}+V_j+\overline{\sum_\mcC\mbR_{\>0}[C_i]}.
\]	
But this is a contradiction to Theorem \ref{thm:dlt-case}. This completes the proof of the first part of the theorem.\\

Let $\alpha$ be a curve class which lies on the $(K_X+\Delta)$-negative portion of the boundary of $\NE(X)_{K_X+\Delta\>0}+\NM(X)$ and that $\alpha$ does not lie on a hyperplane supporting both $\NM(X)$ and $\NE(X)_{K_X+\Delta\>0}$. For a sufficiently small open neighborhood $U$ of $\alpha$ the points of $\overline{U}$ still do not lie on such a hyperplane. We may also assume that $\overline{U}$ is disjoint from $\NE(X)_{K_X+\Delta\>0}$. We define 
\[
	\mcP:=\overline{U}\cap\partial(\NE(X)_{K_X+\Delta\>0}+\NM(X)). 
\]
Fix a compact slice $\mcS$ of $\Eff(X)$ and let $\mcD$ denote the bounding divisors contained in $\mcS$ which have vanishing intersection with some elements of $\mcP$. By construction $\mcD$ is positive on $\NE(X)_{K_X+\Delta\>0}-\{0\}$. By passing to a compact slice, say $\mcT$, it is easy to see that $\mcD$ is also positive on $V_j-\{0\}$ for $j\gg 0$. In other words, every element of $\mcP$ is on the boundary of $\NE(X)_{K_X+\Delta\>0}+V_j+\NM(X)$. By Theorem \ref{thm:dlt-case} there are only finitely many co-extremal rays that lie on this cone, and thus there are only finitely many co-extremal rays contained in $U$. Therefore $\alpha$ can not be an accumulation point.\\

\end{proof}

\section{Finiteness of co-extremal rays in characteristic $0$}
In this section we prove the Batyrev' Conjecture \ref{con:batyrev} over the filed of complex numbers $\mbC$. We follow the same strategy as in the proof of \cite[Theorem 1.3]{Ara10}. For a given $\ve>0$, a pair $(X, \Delta)$ is called a $\ve$-log canonical pair if for every divisor $E$ over $X$ the discrepancies $a(E; X, \Delta)\>-1+\ve$. 
\begin{proposition}\label{pro:boundedness}
	Fix a real number $\ve>0$ and an integer $n>0$. Then there exists a constant $G=G(n, \ve)>0$ depending only on $n$ and $\ve$ and satisfying the following properties:\\
	If $(X, \Delta)$ is a $\mbQ$-factorial projective $\ve$-log canonical pair of dimension $n$ and $K_X+\Delta$ is not pesudo-effective, then for every Mori fiber space $\phi:X\rtmap X'$, $f':X'\to Y'$, obtained via a $(K_X+\Delta)$-MMP, there exists a projective movable curve $C\subset X$ isomorphic (under $\phi$) to a movable curve $C'$ lying on the general fiber of $f'$ such that $-(K_X+\Delta)\cdot C\<G$. 
	\end{proposition}

\begin{proof}
	Since $K_X+\Delta$ is not pseudo-effective, by running a $(K_X+\Delta)$-MMP as in \cite[Corollary 1.3.3]{BCHM10} we end up with a Mori fiber space $f':X'\to Y'$ such that $-(K_{X'}+\Delta')$ is $f'$-ample, where $\phi:X\rtmap X'$ is a birational contraction and $\Delta'=\phi_*\Delta$. Note that $(X', \Delta')$ is a $\mbQ$-factorial projective $\ve$-log canonical pair and $\rho(X'/Y')=1$. \\

 Let $F$ be a general fiber of $f'$. Then $(F, \Delta_F)$ is $\ve$-log canonical and $-(K_F+\Delta_F)$ is ample, where $K_F+\Delta_F=(K_{X'}+\Delta')|_F$. Note that $-K_X=-(K_X+\Delta)+\Delta$ is $f'$-ample, since $\rho(X'/Y')=1$. In particular, $-K_F\sim_\mbQ -K_X|_F$ is ample. By the boundedness of $\ve$-log canonical log-Fano varieties \cite[Theorem 1.1]{Bir16sep}, there exist an integer $M=M(d, \ve)>0$ and a real number $\lambda=\lambda(d, \ve)>0$ depending only on $\dim F=d$ and $\ve$ such that $-MK_F$ is an ample Cartier divisor and $(-K_F)^d\<\lambda$. Then $-MK_F-(K_F+\Delta_F)$ is ample. Thus by Koll\'ar's effective base-point free theorem \cite[Theorem 1.1]{Kol93}, there exists an integer $N=N(M, d)>0$ such that $-NK_F$ is base-point free.\\
 
Now	let $Z'\subset X'$ be the exceptional locus of $\phi^{-1}:X'\rtmap X$. Then $\codim_{X'}Z'\>2$. Let $C'$ be a general curve contained in $F$ obtained by intersecting $(d-1)$ general members of the linear system $|-NK_F|$. Then $C'$ belongs to a moving family of curves dominating $X$, i.e., $C'$ is a movable curve, and $C'$ does not intersect $F\cap Z$, since $\codim_F (F\cap Z)\>2$. In particular, $C'$ can be lifted isomorphically to $X$, we denote the lift by $C$. Then $C$ is a movable curve on $X$, and $-(K_X+\Delta)\cdot C=-(K_{X'}+\Delta')\cdot C'= -(K_F+\Delta_F)\cdot C'\<-K_F\cdot C'=(-K_F)\cdot (-NK_F)^{d-1}=N^{d-1}(-K_F)^d\<N^{d-1}\lambda$. Set $G:=\lambda N^{d-1}$ and we are done.

	\end{proof}



\begin{proof}[Proof of Theorem \ref{thm:nef-cone-finiteness}]
	The first part of the theorem follows either from \cite[Theorem 1.1]{Ara10} or \cite[Theorem 1.3]{Leh12}.\\

Next we reduce the problem to the $\mbQ$-factorial case. Since $(X, \Delta\>0)$ is KLT, there exists a small birational morphism $f:Y\to X$ such that $K_Y+\Delta_Y=f^*(K_X+\Delta)$, and $(Y, \Delta_Y\>0)$ is a $\mbQ$-factorial terminal pair. Assume that the finiteness of co-extremal rays is known on $\mbQ$-factorial KLT pairs. Now $K_Y+\Delta_Y+f^*H$ is KLT for $H$ general ample divisor. Write $\Delta_Y+f^*H\num A+E$ for some ample $\mbR$-divisor $A\>0$ and effective $\mbR$-Cartier divisor $E$. Then $(Y, \Delta_Y+f^*H+\ve(A+E))$ is KLT for all $0<\ve\ll 1$. Then $(Y, (1-\ve)(\Delta_Y+f^*H)+\ve(A+E))$ is KLT. Note that $K_Y+(1-\ve)(\Delta_Y+f^*H)+\ve(A+E)\num K_Y+\Delta_Y+f^*H$. Set $\Delta'=(1-\ve)(\Delta_Y+f^*H)+\ve E$. Then $(Y, \Delta')$ is KLT and $K_Y+\Delta'+\ve A\num K_Y+\Delta_Y+f^*H=f^*(K_X+\Delta+H)$. Therefore by assumption we have
\[
	\overline{NE}(Y)_{K_Y+\Delta'\>0}+\overline{NM}(Y)=\overline{NE}(Y)_{K_Y+\Delta'+f^*H\>0}+\sum_{i=1}^N \mbR_{\>0}[C^Y_i].
\]
Pushing forward these cones by $f_*$ we get the finiteness result on $X$. Therefore replacing $X$ by $Y$ we may assume that $X$ is a $\mbQ$-factorial KLT pair. 	
Let $\Sigma$ be the set of all $(K_X+\Delta)$-negative movable curves classes $[C_i]$ as in Part $(1)$. Let $\Sigma_H\subset\Sigma$ be the set consisting of the classes $[C]\in\Sigma$ such that $(K_X+\Delta+H)\cdot C<0$. Then by \cite[Theorem 1.1]{Ara10} 
\[\overline{NE}(X)_{K_X+\Delta\>0}+\overline{NM}(X)=\overline{NE}(X)_{K_X+\Delta+ H\>0}+\overline{\sum_{[C]\in\Sigma_H} \mbR_{\>0}[C]}.\]
We will show that the set of rays $\left\{\mbR_{\>0}[C]: [C]\in\Sigma_H\right\}$ is finite.\\
Let $\ve>0$ be the minimum log discrepancy of $(X, \Delta)$. Then $(X, \Delta)$ is $\ve$-log canonical. From the statement of \cite[Theorem 1.1]{Ara10} and Proposition \ref{pro:boundedness} we see that the movable curves $C$ in $\Sigma_H$ satisfy the conclusion of the Proposition \ref{pro:boundedness}. In particular, $0<-(K_X+\Delta)\cdot C\<G$ for all $[C]\in\sum_H$. We also have $(K_X+\Delta+\epsilon H)\cdot C<0$. Therefore $\epsilon H\cdot C\<G$ for all $[C]\in\Sigma_H$. In particular, the curves corresponding to the classes $\left\{\mbR_{\>0}[C]: [C]\in\Sigma_H\right\}$ belong a bounded family, and hence they correspond to only  finitely many different numerical equivalence classes.

	\end{proof}

\bibliographystyle{hep}
\bibliography{references.bib}

\end{document}